\documentclass[smallextended,envcountsect]{my_svjour3}

\usepackage{amsmath,amssymb}
\usepackage{a4wide}

\usepackage[misc,geometry]{ifsym}

\begin{document}

\title{Optimal Solutions to Relaxation in Multiple Control Problems of Sobolev 
Type with Nonlocal Nonlinear Fractional Differential Equations}

\titlerunning{Optimal Solutions to Relaxation in Multiple Control Problems of Sobolev Type}

\author{Amar Debbouche \and Juan J. Nieto \and \text{Delfim F. M. Torres}}

\authorrunning{A. Debbouche \and J. J. Nieto \and D. F. M. Torres}

\institute{A. Debbouche \at
Department of Mathematics, Guelma University, Guelma 24000, Algeria\\
\email{amar$_{-}$debbouche@yahoo.fr}
\and
J. J. Nieto \at
Departamento de An\'alisis Matem\'atico,
Universidad de Santiago de Compostela,\\
Santiago de Compostela 15782, Spain\\
\emph{and} Department of Mathematics, Faculty of Science, King Abdulaziz University,
Jeddah, Saudi Arabia\\
\email{juanjose.nieto.roig@usc.es}
\and
D. F. M. Torres \Letter \at
Center for Research and Development in Mathematics and Applications (CIDMA),\\
Department of Mathematics, University of Aveiro, 3810-193 Aveiro, Portugal\\
\email{delfim@ua.pt}
}

\date{Submitted: 26-Dec-2014 / Revised: 14-Apr-2015 / Accepted: 19-Apr-2015}

\maketitle


\begin{abstract}
We introduce the optimality question to the relaxation in multiple control
problems described by Sobolev type nonlinear fractional differential equations
with nonlocal control conditions in Banach spaces. Moreover, we consider
the minimization problem of multi-integral functionals, with integrands
that are not convex in the controls, of control systems
with mixed nonconvex constraints on the controls. We prove,
under appropriate conditions, that the relaxation problem
admits optimal solutions. Furthermore, we show that those optimal
solutions are in fact limits of minimizing sequences of systems
with respect to the trajectory, multi-controls,
and the functional in suitable topologies.


\end{abstract}


\section{Introduction}

The memory and hereditary properties of various materials and processes in electrical circuits,
biology, biomechanics, etc., such as viscoelasticity, electrochemistry, control,
porous media and electromagnetic processes, are widely recognized to be well
predicted by using fractional differential operators \cite{AMA.4,AMA.13,AMA.32,AMA.37,AMA.39,AMA.47}.
During the past decades, the subject of fractional calculus, and its potential applications,
have gained an increase of importance, mainly because it has become
a powerful tool with more accurate and successful results
in modeling several complex phenomena in numerous seemingly
diverse and widespread fields of science and engineering
\cite{AMA.1,AMA.16,AMA.34,AMA.38}.

There has been a significant development in nonlocal problems for (fractional)
differential equations or inclusions (see, for instance,
\cite{AMA.6,AMA.8,AMA.9,AMA.10,AMA.33,AMA.43,AMA.46}).
Indeed, nonlinear fractional differential equations have, in recent years,
been object of an increasing interest because of their wide applicability
in nonlinear oscillations of earthquakes, many physical phenomena such
as seepage flow in porous media, and in fluid dynamic traffic model
\cite{AMA.23,AMA.25,AMA.26}. On the other hand,
there could be no manufacturing, no vehicles,
no computers, and no regulated environment, without control systems.
Control systems are most often based on the principle of feedback,
whereby the signal to be controlled is compared to a desired reference
signal and the discrepancy used to compute corrective control actions \cite{AMA.18}.
Over the last years, one of the fields of science that has been well established
is the fractional calculus of variations (see \cite{AMA.30,AMA.29,AMA.24} and references therein).
Moreover, a generalization of this area, namely the fractional optimal control,
is a topic of research by many authors \cite{AMA.2,AMA.19}.

The fractional optimal control of a distributed system is an optimal control problem
for which the system dynamics is defined with partial fractional differential equations \cite{AMA.35}.
The calculus of variations with constraints being sets of solutions of control systems,
allow us to justify, while performing numerical calculations, the passage
from a nonconvex optimal control problem to the convexified optimal control problem.
We then approximate the latter problem by a sequence of smooth and convex optimal control
problems, for which the optimality conditions are known
and methods of their numerical resolution are well developed.

Sobolev type semilinear equations serve as an abstract formulation
of partial differential equations, which arise in various applications such
as in the flow of fluid through fissured rocks, thermodynamics, and shear
in second order fluids. Further, the fractional differential equations
of Sobolev type appear in the theory of control of dynamical systems,
when the controlled system and/or the controller is described
by a fractional differential equation of Sobolev type. Furthermore,
the mathematical modeling and simulations of systems and processes
are based on the description of their properties in terms of fractional
differential equations of Sobolev type. These new models are more adequate
than previously used integer order models, so fractional order differential
equations of Sobolev type have been investigated by many researchers:
see, for example, Fe\u ckan, Wang and Zhou \cite{AMA.17}
and Li, Liang and Xu \cite{AMA.27}.
In our previous works \cite{AMA.12,AMA.11}, we have introduced the notion
of nonlocal control condition and presented a new kind of Sobolev type condition
that appears in terms of two linear operators. Kamocki \cite{AMA.22} studied
the existence of optimal solutions to fractional optimal control problems.
Liu et al. \cite{AMA.28} established the relaxation for nonconvex
optimal control problems described by fractional differential equations.
Motivated by the above facts and results, we introduce here a new kind of Sobolev type condition
and another form of a nonlocal control condition for nonlinear fractional multiple control systems.
The new Sobolev condition is given in terms of two linear operators and
requires formulating two other characteristic solution operators and their properties,
such as boundedness and compactness. Further, we consider
an optimal control problem $(P)$ of multi-integral functionals, with integrands
that are not convex in the controls. We establish an interrelation between
the solutions of problem $(P)$ and the relaxation problem $(RP)$. Under certain assumptions,
it is proved that $(RP)$ has a solution and that for any solution of $(RP)$
there is a minimizing sequence for $(P)$ converging, in the appropriate topologies,
to the solution of $(RP)$. The convergence takes place simultaneously with respect
to the trajectory, the control and the functional. This property
is usually called relaxation \cite{AMA.15,AMA.31}.

The paper is organized as follows. In Section~\ref{sec:2}, we formulate
and define the problems under study and we review some essential
facts from fractional calculus \cite{AMA.37,AMA.23},
semigroup theory \cite{AMA.36,AMA.44}, and multi-valued analysis \cite{AMA.3,AMA.21},
which are used throughout the work. In Section~\ref{sec:3}, we prove some auxiliary
results that are required for the proof of our main results. Section~\ref{sec:4} deals
with existence results for multiple control systems.
The main results are given in Section~\ref{sec:5}.
We end with Section~\ref{sec:conc} of conclusions.


\section{Preliminaries}
\label{sec:2}

Consider the following nonlocal nonlinear fractional control system of Sobolev type:
\begin{equation}
\label{eq:1.1}
L~^CD_t^{\alpha}[Mx(t)]+Ex(t)=f(t, x(t), B_{1}(t)u_{1}(t),\ldots, B_{r-1}(t)u_{r-1}(t)),~ t\in I
\end{equation}
\begin{equation}
\label{eq:1.2}
x(0)+h(x(t), B_{r}(t)u_{r}(t))=x_{0},
\end{equation}
with mixed nonconvex constraints on the controls
\begin{equation}
\label{eq:1.3}
u_{1}(t),\ldots, u_{r}(t)\in U(t, x(t))~ \text{a.e. on}~ I,
\end{equation}
where $^CD^{\alpha}_{t}$ is the Caputo fractional derivative of order
$\alpha$, $0<\alpha\leq1$, and $t\in I:=[0, a]$. Let $X, Y$ and $Z$
be three Banach spaces such that $Z$ is densely and continuously embedded
in $X$, the unknown function $x(\cdot)$ takes its values in $X$
and $x_{0}\in X$. We assume that the operators $E: D(E)\subset X\rightarrow Y$,
$M: D(M)\subset X\rightarrow Z$, $L: D(L)\subset Z\rightarrow Y$,
and $B_{1},\ldots,B_{r}:I\to\mathcal{L}(T, X)$ are linear and bounded
from $T$ into $X$. The space $T$ is a separable reflexive Banach space
modeling the control space. It is also assumed that $f: I\times X^{r}\rightarrow Y$
and $h: C(X^{2}, X)\rightarrow X$ are given abstract functions, to be specified later,
and $U:I\times X \rightrightarrows 2^T\backslash\{\emptyset\}$ is a multivalued map
with closed values, not necessarily convex.
Let $\widehat{\mathbb{R}} := \ ]-\infty, +\infty]$. For functions
$g_{1},\ldots,g_{r}:I\times X\times T \to \mathbb{R}$, we consider the problem
$$
\max\left\{J_{1},\ldots,J_{r}\Biggl\vert
\begin{array}{lll}
J_{1}(x, u_{1}):=\int_{I}g_{1}(t, x(t), u_{1}(t))dt\\
\qquad \vdots\\
J_{r}(x, u_{r}):=\int_{I}g_{r}(t, x(t), u_{r}(t))dt
\end{array}
\right\}
\longrightarrow \inf \eqno(P)
$$
on solutions of the control system \eqref{eq:1.1}--\eqref{eq:1.2} with constraint \eqref{eq:1.3}.
Let $g_{1, U},\ldots,g_{r, U}: I\times X\times T\to \widehat{\mathbb{R}}$ be the functions defined by
$$
\left\{
\begin{array}{ll}
g_{1,U}(t, x, u_1)
:=\left\{
\begin{array}{ll}
g_{1}(t, x, u_{1}),& \mbox{$u_{1}\in U(t, x)$},\\
+\infty, & \mbox{$u_{1}\notin U(t, x)$},
\end{array}\right.\\
\qquad \vdots\\
g_{r,U}(t, x, u_{r})
:=\left\{
\begin{array}{ll}
g_{r}(t, x, u_{r}),& \mbox{$u_{r}\in U(t, x)$},\\
+\infty, & \mbox{$u_{r}\notin U(t, x)$},
\end{array}\right.
\end{array}\right.
$$
and $g^{**}_{1}(t, x, u_{1}),\ldots,g^{**}_{r}(t, x, u_{r})$
be the bipolar of $u_{1}\to g_{1,U}(t, x, u_{1}),\ldots$,
$u_{r}\to g_{r,U}(t, x, u_{r})$, respectively. Along with problem
$(P)$, we also consider the relaxation problem
$$
\max\left\{J^{**}_{1},\ldots,J^{**}_{r}\Biggl\vert
\begin{array}{lll}
J^{**}_{1}(x, u_{1})=\int_{I}g^{**}_{1}(t, x(t), u_{1}(t))dt\\
\vdots\\
J^{**}_{r}(x, u_{r})=\int_{I}g^{**}_{r}(t, x(t), u_{r}(t))dt
\end{array}
\right\}
\longrightarrow \inf \eqno(RP)
$$
on the solutions of control system \eqref{eq:1.1}--\eqref{eq:1.2}
with the convexified constraints
\begin{equation}
\label{eq:1.4}
u_{1}(t),\ldots, u_{r}(t)\in \operatorname{cl} \operatorname{conv}U(t,x(t))\quad\text{a.e. on }I
\end{equation}
on the controls, where $conv$ denote the convex hull and $cl$ the closure.
In our results, we will denote by $\mathcal{R}_U$ and $\mathcal{T}r_{U},
(\mathcal{R}_{\operatorname{cl} \operatorname{conv}U}$
and $\mathcal{T}r_{\operatorname{cl} \operatorname{conv} U})$ the sets of all solutions and
all trajectories of control system \eqref{eq:1.1}--\eqref{eq:1.3}
(control system \eqref{eq:1.1}--\eqref{eq:1.2},\eqref{eq:1.4}, respectively).

\begin{definition}
\label{Definition 2.1}
The fractional integral of order $\alpha>0$ of a function $f\in L^{1}([a,b],\mathbb{R})$ is given by
$$
I^{\alpha}_{a}f(t):=\frac{1}{\Gamma(\alpha)}\int_{a}^{t}(t-s)^{\alpha-1}f(s)ds,
$$
where $\Gamma$ is the classical gamma function.
\end{definition}

If $a=0$, then we can write $I^{\alpha}f(t) := (g_{\alpha}*f)(t)$, where
$$
g_{\alpha}(t):=
\left\{
\begin{array}{ll}
\frac{1}{\Gamma(\alpha)}t^{\alpha-1},& \mbox{$t>0$},\\
0, & \mbox{$t\leq 0$}
\end{array}\right.
$$
and, as usual, $*$ denotes convolution. Moreover,
$\lim\limits_{\alpha \downarrow 0} g_{\alpha}(t)=\delta(t)$
with $\delta$ the delta Dirac function.

\begin{definition}
\label{Definition 2.2}
The Riemann--Liouville fractional derivative of order $\alpha>0$,
$n-1<\alpha<n$, $n\in \mathbb{N}$, is given by
$$
^{L}D^{\alpha}f(t) := \frac{1}{\Gamma(n-\alpha)}\frac{d^{n}}{dt^{n}}
\int_{0}^{t}\frac{f(s)}{(t-s)^{\alpha+1-n}}ds,~
t>0,
$$
where function $f$ has absolutely continuous derivatives up to order $(n-1)$.
\end{definition}

\begin{definition}
\label{Definition 2.3}
The Caputo fractional derivative of order $\alpha>0$,
$n-1<\alpha<n$, $n\in \mathbb{N}$, is given by
$$
^{C}D^{\alpha}f(t) := ~^{L}D^{\alpha}\left(f(t)
-\sum\limits_{k=0}^{n-1}\frac{t^{k}}{k!}f^{(k)}(0)\right),
\quad t>0,
$$
where function $f$ has absolutely continuous derivatives up to order $(n-1)$.
\end{definition}

If $f$ is an abstract function with values in $X$, then the integrals
that appear in Definitions~\ref{Definition 2.1} to \ref{Definition 2.3}
are taken in Bochner's sense.

\begin{remark}
\label{Remark 2.1}
Let $n-1<\alpha<n$, $n\in \mathbb{N}$.
The following properties hold:
\begin{itemize}
\item [(i)] If $f\in C^{n}([0, \infty[)$, then
$$
^{C}D^{\alpha}f(t)=\frac{1}{\Gamma(n-\alpha)}
\int_{0}^{t}\frac{f^{(n)}(s)}{(t-s)^{\alpha+1-n}}ds
=I^{n-\alpha}f^{(n)}(t),
\quad t>0;
$$
\item [(ii)] The Caputo derivative of a constant function is equal to zero;

\item [(iii)] The Riemann--Liouville derivative of a constant function is given by
$$
^{L}D^{\alpha}_{a^{+}}C=\frac{C}{\Gamma(1-\alpha)}(t-a)^{-\alpha},~ 0<\alpha<1.
$$
\end{itemize}
\end{remark}

We make the following assumptions:
\begin{itemize}
\item [(H$_1$)] $L: D(L)\subset Z\rightarrow Y$ and $M: D(M)\subset X\rightarrow Z$
are linear operators, and $E: D(E)\subset X\rightarrow Y$ is closed.

\item [(H$_2$)] $D(M)\subset D(E)$, $\text{Im}(M)\subset D(L)$ and $L$ and $M$ are bijective.

\item [(H$_3$)] $L^{-1}: Y\rightarrow D(L)\subset Z$
and $M^{-1}: Z\rightarrow D(M)\subset X$ are linear, bounded and compact operators.
\end{itemize}
Note that (H$_3$) implies that $L$ and $M$ are closed.
Indeed, if $L^{-1}$ and $M^{-1}$ are closed and injective, then
their inverse are also closed. From (H$_1$)--(H$_3$) and the closed graph theorem,
we obtain the boundedness of the linear operator $L^{-1}EM^{-1}: Z\rightarrow Z$. Consequently,
$L^{-1}EM^{-1}$ generates a semigroup $\lbrace Q(t), t\geq0\rbrace$, $Q(t):=e^{L^{-1}EM^{-1}t}$.
We assume that $M_{0}:=\sup_{t\geq0}\Vert Q(t)\Vert<\infty$ and, for short,
we denote $C_{1}:=\Vert L^{-1}\Vert$ and $C_{2}:=\Vert M^{-1}\Vert$.
According to previous definitions, it is suitable to rewrite problem
\eqref{eq:1.1}--\eqref{eq:1.2} as the equivalent integral equation
\begin{multline}
\label{eq:2.1}
Mx(t)=Mx(0)\\
+\frac{1}{\Gamma(\alpha)}\int_{0}^{t}(t-s)^{\alpha-1}
[-L^{-1}Ex(s)+L^{-1}f(s, x(s), B_{1}(s)u_{1}(s),\ldots, B_{r-1}(s)u_{r-1}(s))]ds,
\end{multline}
provided the integral in \eqref{eq:2.1} exists a.e. in $t\in J$.
Before formulating the definition of mild solution of system
\eqref{eq:1.1}--\eqref{eq:1.3}, we first introduce some necessary notions.
Let $I:=[0,a]$ be a closed interval of the real line with the Lebesgue
measure $\mu$ and the $\sigma$-algebra $\Sigma$ of $\mu$ measurable sets.
The norm of the space $X$ (or $T$) will be denoted by $\|\cdot\|_X$
(or $\|\cdot\|_T$). We denote by $C(I,X)$ the space of all continuous
functions from $I$ into $X$ with the supnorm given by
$\|x\|_{C}:=\sup_{t\in I}\|x(t)\|_X$ for $x\in C(I,X)$.
For any Banach space $V$, the symbol $\omega$-$V$ stands for $V$ equipped
with the weak topology $\sigma(V,V^*)$. The same notation will be used
for subsets of $V$. In all other cases, we assume that $V$ and its subsets
are equipped with the strong (normed) topology.

Throughout the paper, $A:=-L^{-1}EM^{-1}: D(A)\subset Z\rightarrow Z$
is the infinitesimal generator of a compact analytic
semigroup of uniformly bounded linear operators $Q(\cdot)$ in $X$.
Then, there exists a constant $M_{0}\geq1$ such that $\Vert Q(t)\Vert\leq M_{0}$
for $t\geq0$. The operators $B_{i}\in L^{\infty}(I, \mathcal{L}(T, X))$,
and we let $\Vert B_{i}\Vert$ stand for $\Vert B_{i}\Vert_{L^{\infty}(I, \mathcal{L}(T, X))}$.

We now proceed with some basic definitions and results from multivalued analysis.
For more details on multivalued analysis
we refer to the books \cite{AMA.3,AMA.21}. We use the following symbols:
$P_{f}(T)$ is the set of all nonempty closed
subsets of $T$; $P_{bf}(T)$ is the set of all nonempty, closed and bounded
subsets of $T$. On $P_{bf}(T)$, we have a metric, known as the Hausdorff metric,
defined by
\[
d_{H}(A,B):=\max\left\lbrace\sup_{a\in A}d(a,B),\,\sup_{b\in B}d(b,A)\right\rbrace,
\]
where $d(x,C)$ is the distance from a point $x$ to a set $C$.
We say that a multivalued map is $H$-continuous if it is continuous in
the Hausdorff metric $d_{H}(\cdot,\cdot)$.
Let $F: I \rightrightarrows 2^{T}\backslash\{\emptyset\}$ be a multifunction.
For $1\leq p\leq +\infty$, we define
$S^{p}_{F}:=\lbrace f\in L^{p}(I, T): f(t)\in F(t)$ a.e. on $I\rbrace$.
We say that a multivalued map $F:I \rightrightarrows P_f(T)$ is measurable if
$F^{-1}(E)=\{t\in I:F(t)\cap E\neq\emptyset\}\in \Sigma$ for every closed
set $E\subseteq T$. If $F:I\times T\to P_f(T)$, then the measurability
of $F$ means that $F^{-1}(E)\in\Sigma\otimes\mathcal{B}_{T}$,
where $\Sigma\otimes\mathcal{B}_{T}$ is the $\sigma$-algebra of subsets in
$I\times T$ generated by the sets $A\times B$, $A\in\Sigma$,
$B\in\mathcal{B}_{T}$, and $\mathcal{B}_{T}$ is the $\sigma$-algebra of
the Borel sets in $T$.

Suppose that $V_{1}$ and $V_{2}$ are two Hausdorff topological spaces and
$F: V_{1}\to 2^{V_{2}} \backslash\{\emptyset\}$. We say that $F$
is lower semicontinuous in the sense of Vietoris (l.s.c., for short)
at a point $x_0\in V_{1}$, if for any open set $W\subseteq V_{2}$,
$F(x_0)\cap W\neq\emptyset$, there is a neighborhood
$O(x_0)$ of $x_0$ such that $F(x)\cap W\neq\emptyset$
for all $x\in O(x_0)$. Similarly, $F$ is said to be upper semicontinuous in the sense
of Vietoris (u.s.c., for short) at a point $x_0\in V_{1}$, if for any open set
$W\subseteq V_{2}$, $F(x_0)\subseteq W$, there is a neighborhood $O(x_0)$ of $x_0$
such that $F(x)\subseteq W$ for all $x\in O(x_0)$. For more properties of
l.s.c and u.s.c, we refer to the book \cite{AMA.21}.
Besides the standard norm on $L^q(I, T)$ (here, $T$ is a separable reflexive
Banach space), $1<q<\infty$, we also consider the so called weak norm:
\begin{equation}
\label{eq:2.2}
\|u_{i}(\cdot)\|_{\omega}:=\sup_{0\leq t_1\leq t_2\leq a}
\left\Vert\int_{t_1}^{t_2}u_{i}(s)ds\right\Vert_T,~  u_{i}\in L^q(I, T),
\ i=1,\ldots,r.
\end{equation}
The space $L^q(I, T)$ furnished with this norm will be denoted by
$L_{\omega}^q(I, T)$. The following result gives a relation between
convergence in $\omega$-$L^q(I, T)$ and convergence in $L_{\omega}^q(I, T)$.

\begin{lemma}(see \cite{AMA.41})
\label{Lemma 2.1}
If sequences $\{u_{1,n}\}_{n\geq1},\ldots,\{u_{r,n}\}_{n\geq1}\subseteq L^q(I, T)$
are bounded and converge to $u_{1},\ldots,u_{r}$ in $L_{\omega}^q(I, T)$,
respectively, then they converge to $u_{1},\ldots,u_{r}$
in $\omega$-$L^q(I, T)$, respectively.
\end{lemma}

We make use of the following assumptions on the data of our problems.
\begin{itemize}
\item[(H1)] The nonlinear function $f: I\times X^{r}\to Y$ satisfies the following:
\begin{itemize}
\item[(1)] $t\to f(t, x_{1},\ldots,x_{r})$ is measurable
for all $(x_{1},\dots,x_{r})\in X^{r}$;

\item[(2)] $\Vert f(t, x_{1},\ldots,x_{r})-f(t, y_{1},\ldots,y_{r})\Vert_{Y}\leq k_{1}(t)
\sum_{i=1}^{r}\Vert x_{i}-y_{i}\Vert_{X}$ a.e. on $I$, $k_{1}\in L^{\infty}(I,\mathbb{R}^+)$;

\item[(3)] there exists a constant $0<\beta<\alpha$ such that
$\|f(t, x_{1},\ldots,x_{r})\|_Y\leq a_1(t)+c_1\sum_{i=1}^{r}\Vert x_{i}\Vert_{X}$
a.e. in $t\in I$, where $a_1\in L^{1/\beta}(I, \mathbb{R}^+)$ and $c_1>0$.
\end{itemize}

\item[(H2)] The nonlocal function $h: C(J: X, X)\rightarrow X$ satisfies the following:
\begin{itemize}
\item[(1)] $t\to h(x, y)$ is measurable for all $x,y\in X$;

\item[(2)] $\|h(x_{1}, y_{1})-h(x_{2}, y_{2})\|_X\leq k_{2}(t)\lbrace\|x_{1}-x_{2}\|_X+\|y_{1}-y_{2}\|_X\rbrace$
a.e. on $I$, $k_{2}\in L^{\infty}(I,\mathbb{R}^+)$;

\item[(3)] there exists a constant $0<\beta<\alpha$ such that
$\|h(x, y)\|_X\leq a_2(t)+c_2\lbrace\|x\|_X+\|y\|_X\rbrace$
a.e. in $t\in I$ and all $x,y\in X$, where
$a_2\in L^{1/\beta}(\mathbb{R}^+)$ and $c_2>0$.
\end{itemize}

\item[(H3)] The multivalued map $U: I\times X
\rightrightarrows P_{f}(T)$ is such that:
\begin{itemize}
\item[(1)] $t\to U(t,x)$ is measurable for all $x\in X$;

\item[(2)] $d_H(U(t, x),U(t, y))\leq k_3(t)\|x-y\|_X$ a.e. on $I$,
$k_3\in L^{\infty}(I,\mathbb{R}^+)$;

\item[(3)] there exists a constant $0<\beta<\alpha$ such that
$$
\|U(t, x)\|_T=\sup\{\|v\|_T: v\in U(t, x)\}
\leq a_3(t)+c_3\|x\|_X \quad \text{ a.e. in } t\in I,
$$
where $a_3\in L^{1/\beta}(I, \mathbb{R}^+)$ and $c_3>0$.
\end{itemize}

\item[(H4)] Functions $g_{i}: I\times X\times T\to \mathbb{R}$,
$i=1,\ldots,r$, are such that:
\begin{itemize}
\item[(1)] the map $t\to g_{i}(t, x, u_{i})$ is measurable
for all $(x, u_{i})\in X\times T$;

\item[(2)] $\vert g_{i}(t, x ,u_{i})-g_{i}(t, y ,v_{i})\vert
\leq k^{\prime}_{4}(t)\|x-y\|_X +k^{\prime\prime}_{4}\Vert u_{i}-v_{i}\Vert_{T}$
a.e., $k^{\prime}_{4}\in L^{1}(I,\mathbb{R}^+)$, $k^{\prime\prime}_{4}>0$;

\item[(3)] $\vert g_{i}(t, x ,u_{i})\vert\leq a_{4}(t)+b_{4}(t)\Vert x\Vert_{X}
+c_{4}\Vert u_{i}\Vert_{T}$ a.e. $t\in I,
a_4,b_{4}\in L^{1/\beta}(I, \mathbb{R}^+)$, $c_{4}>0$.
\end{itemize}
\end{itemize}

\begin{definition}
\label{Definition 2.4}
A solution of the control system \eqref{eq:1.1}--\eqref{eq:1.3} is defined
to be a vector of functions $(x(\cdot), u_{1}(\cdot),\ldots,u_{r}(\cdot))$
consisting of a trajectory $x\in C(I, X)$ and $r$ multiple controls
$u_{1},\ldots,u_{r}$ $\in L^{1}(I, T)$ satisfying system \eqref{eq:1.1}--\eqref{eq:1.2}
and the inclusion \eqref{eq:1.3} almost everywhere.
\end{definition}

A solution of control system \eqref{eq:1.1}--\eqref{eq:1.2},
\eqref{eq:1.4} can be defined similarly.

\begin{definition}(see \cite{AMA.46,AMA.12,AMA.45})
\label{Definition 2.5}
A vector of functions $(x,u_{1},\ldots,u_{r})$ is a mild solution of the control
system \eqref{eq:1.1}--\eqref{eq:1.3} iff $x\in C(I,X)$ and there exist
$u_{1},\,\ldots,\,u_{r} \in L^1(I,T)$ such that
$u_{1}(t),\,\ldots,\,u_{r}(t)$ $\in U(t,x(t))$
a.e. in $t\in I$, $x(0)=x_0-h(x(t), B_{r}(t)u_{r}(t))$,
and the following integral equation is satisfied:
\begin{multline*}
x(t)=S_{\alpha}(t)M[x_{0}-h(x(t), B_{r}(t)u_{r}(t))]\\
+\int_{0}^{t}(t-s)^{\alpha-1}T_{\alpha}(t-s)L^{-1}
f(s, x(s), B_{1}(s)u_{1}(s),\ldots, B_{r-1}(s)u_{r-1}(s))ds,
\end{multline*}
where
$$
S_{\alpha}(t):=\int_{0}^{\infty}M^{-1}\zeta_{\alpha}(\theta)Q(t^{\alpha}\theta)d\theta,
\quad T_{\alpha}(t):=\alpha\int_{0}^{\infty}M^{-1}\theta\zeta_{\alpha}(\theta)
Q(t^{\alpha}\theta)d\theta,
$$
$$
\zeta_{\alpha}(\theta):=\frac{1}{\alpha}\theta^{-1-\frac{1}{\alpha}}
\varpi_{\alpha}(\theta^{-\frac{1}{\alpha}})\geq0,
\quad \varpi_{\alpha}(\theta):=\frac{1}{\pi}\sum_{n=1}^{\infty}(-1)^{n-1}
\theta^{-\alpha n-1}\frac{\Gamma(n\alpha+1)}{n!}\sin (n\pi\alpha),
\ \theta \in ]0, \infty[,
$$
with $\zeta_{\alpha}$ the probability density function defined on $]0, \infty[$,
that is, $\zeta_{\alpha}(\theta)\geq 0$, $\theta\in ]0, \infty[$,
and $\int_{0}^{\infty}\zeta_{\alpha}(\theta)d\theta=1$.
\end{definition}

A similar definition can be introduced for the control
system \eqref{eq:1.1}--\eqref{eq:1.2},\eqref{eq:1.4}.

\begin{remark}(see \cite{AMA.45})
\label{Remark 2.2}
One has
$\displaystyle \int_0^{\infty}\theta\xi_{\alpha}(\theta)d\theta
=\frac{1}{\Gamma(1+\alpha)}$.
\end{remark}

\begin{lemma}(see \cite{AMA.45})
\label{Lemma 2.2}
The characteristic operators $S_{\alpha}$ and $T_{\alpha}$
have the following properties:
\begin{itemize}
\item[(1)] for any fixed $t\geq0$, $S_{\alpha}(t)$ and $T_{\alpha}(t)$
are linear and bounded operators, i.e., for any $x\in X$,
\[
\|S_{\alpha}(t)x\|_X\leq C_{2}M_0\|x\|_X,\quad \|T_{\alpha}(t)x\|_X
\leq\frac{C_{2}M_0}{\Gamma(\alpha)}\|x\|_X;
\]

\item[(2)] $\{S_{\alpha}(t),t\geq0\}$ and $\{T_{\alpha}(t),t\geq0\}$
are strongly continuous;

\item[(3)] for every $t>0$, $S_{\alpha}(t)$ and $T_{\alpha}(t)$
are compact operators.
\end{itemize}
\end{lemma}

\begin{lemma}(see \cite{AMA.14})
\label{Lemma 2.3}
Let $x(t)$ be continuous and non-negative on $[0, a]$. If
$$
x(t)\leq \psi(t)+\lambda\int_{0}^{t}\frac{x(s)}{(t-s)^{\gamma}}ds,
\quad 0\leq t\leq a,
$$
where $0\leq\gamma<1$, $\psi(t)$ is a non-negative monotonic increasing
continuous function on $[0, a]$, and $\lambda$ is a positive constant, then
$$
x(t)\leq \psi(t)E_{1-\gamma}(\lambda\Gamma(1-\gamma)t^{1-\gamma}),
\quad 0\leq t\leq a,
$$
where $E_{1-\gamma}(z)$ is the Mittag--Leffler function defined for all $\gamma<1$ by
$$
E_{1-\gamma}(z):=\sum\limits_{n=0}^{\infty}\frac{z^{n}}{\Gamma(n(1-\gamma)+1)}.
$$
\end{lemma}


\section{Auxiliary Results}
\label{sec:3}

In this section we give some auxiliary results, which are required
for the proof of our main results. We begin with a prior
estimation of the trajectory of the control system.

\begin{lemma}
\label{Lemma 3.1}
For any admissible trajectory $x$ of the control system
\eqref{eq:1.1}--\eqref{eq:1.2},\eqref{eq:1.4}, that is,
for any $x\in\mathcal{T}r_{\operatorname{cl} \operatorname{conv}U}$,
there is a constant $L_{0}$ such that
\begin{equation}
\label{eq:3.1}
\|x\|_C\leq L_0.
\end{equation}
\end{lemma}

\begin{proof}
From Definition~\ref{Definition 2.5}, there exist
$u_{1}(t),\ldots,u_{r}(t)\in \operatorname{cl} \operatorname{conv} U(t,x(t))$
a.e. in $t\in I$ for any $x\in\mathcal{T}r_{\operatorname{cl} \operatorname{conv}U}$, and
\begin{align*}
x(t)&=S_{\alpha}(t)M[x_{0}-h(x(t), B_{r}(t)u_{r}(t))]\\
&+\int_{0}^{t}(t-s)^{\alpha-1}T_{\alpha}(t-s)L^{-1}f(s, x(s),
B_{1}(s)u_{1}(s),\ldots, B_{r-1}(s)u_{r-1}(s))ds.
\end{align*}
Then, by Lemma~\ref{Lemma 2.2}, (H1.3), (H2.3), (H3.3), and H\"older's inequality, one gets
\begin{align*}
\|x(t)\|_X&\leq C_{2}M_0\Vert M\Vert\left\lbrace\|x_0\|_X
+ a_{2}(t)+c_2[\Vert x\Vert_X +\Vert B_r\Vert(a_3(t)+c_3\Vert x\Vert_X)]\right\rbrace\\
&\quad +\frac{C_{1}C_{2}M_0}{\Gamma(\alpha)}
\int_0^t(t-s)^{\alpha-1}\left\lbrace a_1(s)+c_1[\Vert x\Vert_X
+\sum_{i=1}^{r-1}\Vert B_i\Vert(a_3(s)+c_3\Vert x\Vert_X)]\right\rbrace ds\\
&\leq C_{2}M_0\Vert M\Vert\left\lbrace\|x_0\|_X+ a_{2}(t)
+c_2[\Vert x\Vert_X +\Vert B_r\Vert(a_3(t)+c_3\Vert x\Vert_X)]\right\rbrace\\
&\quad +\frac{C_{1}C_{2}M_0}{\Gamma(\alpha)}\left[
\frac{(1-\beta)}{\alpha-\beta}a^{\frac{\alpha-\beta}{1-\beta}}\right]^{1-\beta}
\left[\Vert a_1\Vert_{L^{\frac{1}{\beta}}}+c_1
\sum_{i=1}^{r-1}\Vert B_i\Vert\Vert a_3\Vert_{L^{\frac{1}{\beta}}}\right]\\
&\quad +\frac{C_{1}C_{2}M_0}{\Gamma(\alpha)}\left(c_1+c_1c_3
\sum_{i=1}^{r-1}\Vert B_i\Vert\right)
\int_0^t(t-s)^{\alpha-1}\Vert x\Vert_X ds.
\end{align*}
From the above inequality, using the well-known singular-version of Gronwall's
inequality (see Lemma~\ref{Lemma 2.3}), we can deduce that the inequality
\eqref{eq:3.1} is satisfied, that is, there exists a constant $L_{0}>0$
such that $\|x\|_C\leq L_{0}$.
\hfill $\Box$
\end{proof}

Let $\operatorname{pr}_{L_0}:X\to X$ be a $L_0$-radial retraction, that is,
\[
\operatorname{pr}_{L_0}(x)
:=
\begin{cases}
x, & \|x\|_X\leq L_0,\\
\frac{L_0x}{\|x\|_X}, & \|x\|_X>L_0.
\end{cases}
\]
This map is Lipschitz continuous. We define $U_1(t,x):= U(t,\operatorname{pr}_{L_0}x)$.
Evidently, $U_1$ satisfies (H3.1) and (H3.2). Moreover, by the properties
of $\operatorname{pr}_{L_0}$, we have, a.e. in $t\in I$, all $x\in X$ and
all $u_{1},\ldots,u_{r}\in U_1(t, x)$, that
\[
\sup\lbrace\|u_{1}\|_T,\dots,\|u_{r}\|_T\rbrace\leq a_3(t)+c_3L_0\,\,
\textrm {and}\,\,\sup\lbrace\|u_{1}\|_T,
\dots,\|u_{r}\|_T\rbrace\leq a_3(t)+c_3\|x\|_X.
\]
Hence, Lemma~\ref{Lemma 3.1} is still valid with $U(t, x)$
substituted by $U_1(t,x)$. Consequently, without loss of generality,
we assume that, a.e. in $t\in I$ and all $x\in X$,
\begin{equation}
\label{eq:3.2}
\sup\{\|v\|_T:v\in U(t, x)\}\leq\varphi(t)=a_3(t)+c_3L_0,
~ \text{with}~ \varphi\in L^{1/\beta}(I, \mathbb{R}^+).
\end{equation}
Now we consider the following fractional nonlocal semilinear
auxiliary problem of Sobolev type:
\begin{equation}
\label{eq:3.3}
L~^CD_t^{\alpha}[Mx(t)]+Ex(t)=f(t, x(t)),~ t\in I,
\end{equation}
\begin{equation}
\label{eq:3.4}
x(0)=h(x(t)).
\end{equation}
It is clear that, for every $f\in L^{1/\beta}(I\times X, Y)$, $h\in L^{1/\beta}(I:X, X)$,
$0<\beta<\alpha$, the problem \eqref{eq:3.3}--\eqref{eq:3.4}
has a unique mild solution $H(f, h)\in C(I, X)$, which
is given by
\[
H(f,h)(t)=S_{\alpha}(t)Mh(x(t))+\int_{0}^{t}(t-s)^{\alpha-1}T_{\alpha}(t-s)L^{-1}f(s, x(s))ds.
\]
Let $\varphi$ be defined by \eqref{eq:3.2}. We put
\begin{gather*}
T_{\varphi}=\lbrace u_{i}\in L^{1/\beta}(I, T): \Vert u_{i}(t)\Vert\leq\varphi(t)
~ \text{a.e.}~ t\in I, i=1,\ldots,r \rbrace,\\
X_{\varphi}=\lbrace f, h\vert f\in L^{1/\beta}(I\times X, Y),
h\in L^{1/\beta}(I:X, X)\rbrace.
\end{gather*}
The following lemma gives a property of the solution map
$S: T_{\varphi}\to C(I, X)$ of \eqref{eq:1.1}--\eqref{eq:1.2},
which is crucial in our investigation.

\begin{lemma}
\label{Lemma 3.2}
The solution map $S: T_{\varphi}\to C(I, X)$ is continuous from
$\omega$-$T_{\varphi}$ into $C(I, X)$.
\end{lemma}

\begin{proof}
The operator $H: L^{1/\beta}(I\times X, Y)\times L^{1/\beta}(I:X, X)\to C(I,X)$
is linear. The estimation
$$
\|H(f,h)\|_C\leq C_{2}M_0\Vert M\Vert\Vert h\Vert_{L^{\frac{1}{\beta}}}
+\frac{C_{1}C_{2}M_0}{\Gamma(\alpha)}\left[\frac{(1-\beta)}{\alpha-\beta}
a^{\frac{\alpha-\beta}{1-\beta}}\right]^{1-\beta}
\Vert f\Vert_{L^{\frac{1}{\beta}}}
$$
shows that $H$ is continuous. Hence, $H$ is also continuous from
$\omega$-$L^{1/\beta}(I\times X, Y)\times L^{1/\beta}(I:X, X)$ to
$\omega$-$C(I,X)$.
Let $C\in P_{b}(L^{1/\beta}(J,X))$ and suppose that for any
$f,h\in C$, $\|f\|_{L^{1/\beta}(I\times X, Y)}\leq K_{1}$
and $\|h\|_{L^{1/\beta}(I:X, X)}\leq K_{2}$ ($K_{1}, K_{2}>0$).
Next we will show that $H$ is completely continuous.

\noindent \textit{Step 1.} From Lemma~\ref{Lemma 3.1},
 we have that the map $\|H(f,h)(t)\|_X$ is
uniformly bounded.

\noindent \textit{Step 2.} $H$ is equicontinuous on $C$.
Let $0\leq t_1<t_2\leq a$.
For any $f,h\in C$, we obtain
\begin{align*}
\|H(f,h)(t_2)-H(f,h)(t_1)\|_X
&\leq \left\Vert [S_{\alpha}(t_{2})-S_{\alpha}(t_{1})]M h(x)\right\Vert_{X}\\
&+\left\Vert\int_{t_1}^{t_2}(t_2-s)^{\alpha-1}T_{\alpha}(t_2-s)L^{-1}f(s, x(s))ds
\right\Vert_X \\
&+\left\Vert\int_0^{t_1}\Big[(t_2-s)^{\alpha-1}-(t_1-s)^{\alpha-1}\Big]
 T_{\alpha}(t_2-s)L^{-1}f(s, x(s))ds\right\Vert_X  \\
&+\left\Vert\int_0^{t_1}(t_1-s)^{\alpha-1}\Big[T_{\alpha}(t_2-s)
 -T_{\alpha}(t_1-s)\Big]L^{-1}f(s, x(s))ds\right\Vert_X\\
&=: I_1+I_2+I_3+I_4.
\end{align*}
By using analogous arguments as in Lemma~\ref{Lemma 3.1}, we have
\begin{align*}
I_1&\leq K_{2}\Vert M\Vert \sup\|S_{\alpha}(t_2)-S_{\alpha}(t_1)\|,\\
I_2&\leq \frac{C_1C_2M_0K_{1}}{\Gamma(\alpha)}
\Big[\frac{1-\beta}{\alpha-\beta}\Big]^{1-\beta}(t_2-t_1)^{\alpha-\beta}, \\
I_3&\leq \frac{C_1C_2M_0K_{1}}{\Gamma(\alpha)}\Big(\int_0^{t_1}
\big((t_1-s)^{\alpha-1}-(t_2-s)^{\alpha-1}\big)^{1/(1-\beta)}ds\Big)^{1-\beta}\\
&\leq \frac{C_1C_2M_0K_{1}}{\Gamma(\alpha)}\Big(\int_0^{t_1}
\big((t_1-s)^{\frac{\alpha-1}{1-\beta}}-(t_2-s)^{\frac{\alpha-1}{1-\beta}}\big)ds\Big)^{1-\beta}\\
&= \frac{C_1C_2M_0K_{1}}{\Gamma(\alpha)}
\Big[\frac{1-\beta}{\alpha-\beta}\Big]^{1-\beta}
\Big(t_1^{\frac{\alpha-\beta}{1-\beta}}
-t_2^{\frac{\alpha-\beta}{1-\beta}}+(t_2-t_1)^{\frac{\alpha-\beta}{1-\beta}}
\Big)^{1-\beta}\\
&\leq \frac{C_1C_2M_0K_{1}}{\Gamma(\alpha)}
\Big[\frac{1-\beta}{\alpha-\beta}\Big]^{1-\beta}
\big(t_2-t_1\big)^{\alpha-\beta}.
\end{align*}
For $t_1=0$ and $0<t_2\leq b$, it is easy to see that $I_4=0$. For $t_1>0$ and
$\epsilon>0$ small enough,
\begin{align*}
I_4 &\leq \left\Vert\int_0^{t_1-\epsilon}(t_1-s)^{\alpha-1}\Big(T_{\alpha}(t_2-s)
-T_{\alpha}(t_1-s)\Big)L^{-1}f(s, x(s))ds\right\Vert_X \\
&\quad +\left\Vert\int_{t_1-\epsilon}^{t_1}(t_1-s)^{\alpha-1}\Big(T_{\alpha}(t_2-s)
-T_{\alpha}(t_1-s)\Big)L^{-1}f(s, x(s))ds\right\Vert_X \\
&\leq \sup_{s\in[0,t_1-\epsilon]}\|T_{\alpha}(t_2-s)
-T_{\alpha}(t_1-s)\|C_1 K_{1}\Big[\frac{1-\beta}{\alpha-\beta}\Big]^{1-\beta}
\left(t_1^{\frac{\alpha-\beta}{1-\beta}}-\epsilon^{\frac{\alpha-\beta}{1-\beta}}\right)^{1-\beta}\\
&\quad +\frac{2C_1C_2M_0K_{1}}{\Gamma(\alpha)}
\Big[\frac{1-\beta}{\alpha-\beta}\Big]^{1-\beta}\epsilon^{\alpha-\beta}.
\end{align*}
Combining the estimations for $I_1$, $I_2$, $I_3$ and $I_4$, and letting
$t_2\to t_1$ and $\epsilon\to 0$ in $I_4$, we conclude that $H$ is equicontinuous.
For more details see \cite{AMA.46}.

\noindent \textit{Step 3.} The set $\Pi(t):=\{H(f,h)(t): f,h\in C\}$ is relatively compact in $X$.
Clearly, $\Pi(0)$ is compact. Hence, it is only necessary to
consider $t>0$. For each $g \in ]0,t[$, $t\in ]0,a]$, $f,h\in C$, and $\delta>0$
being arbitrary, we define
$\Pi_{g,\delta}(t):=\{H_{g,\delta}(f,h)(t): f,h\in C\}$, where
\begin{align*}
H_{g,\delta}(f,h)(t)
&=\int_{\delta}^{\infty}M^{-1}\xi_{\alpha}(\theta)Q(t^{\alpha}\theta) Mh(x)d\theta\\
&\quad + \alpha\int_0^{t-g}\int_{\delta}^{\infty}\theta(t-s)^{\alpha-1}\xi_{\alpha}
(\theta) Q((t-s)^{\alpha}\theta)L^{-1}f(s, x(s))d\theta ds \\
&=Q(g^{\alpha}\delta)\int_{\delta}^{\infty}M^{-1}\xi_{\alpha}(\theta)
Q(t^{\alpha}\theta-g^{\alpha}\delta)Mh(x)d\theta\\
&\quad +\alpha Q(g^{\alpha}\delta)\int_0^{t-g}\int_{\delta}^{\infty}
\theta(t-s)^{\alpha-1}\xi_{\alpha}(\theta)
Q\big((t-s)^{\alpha}\theta-g^{\alpha}\delta\big)L^{-1}f(s, x(s))d\theta ds\\
&:=Q(g^{\alpha}\delta)y(t, g).
\end{align*}
Because $Q(g^{\alpha}\delta)$ is compact and $y(t, g)$ is bounded, we obtain
that the set $\Pi_{g,\delta}(t)$ is relatively compact in $X$ for any
$g \in ]0,t[$ and $\delta>0$. Moreover, we have
\begin{align*}
&\|H(f,h)(t)-H_{g,\delta}(f,h)(t)\|_X=\biggl\Vert
\int_{0}^{\delta}M^{-1}\xi_{\alpha}(\theta)Q(t^{\alpha}\theta)Mh(x)d\theta\\
&\quad +\alpha\int_0^t\int_0^{\delta}M^{-1}\theta(t-s)^{\alpha-1}
\xi_{\alpha}(\theta) Q((t-s)^{\alpha}\theta)L^{-1}f(s, x(s))d\theta ds \\
&\quad +\alpha\int_0^t\int_{\delta}^{\infty}M^{-1}\theta(t-s)^{\alpha-1}\xi_{\alpha}(\theta)
 Q((t-s)^{\alpha}\theta)L^{-1}f(s, x(s))d\theta ds \\
&\quad -\alpha\int_0^{t-g}\int_{\delta}^{\infty}M^{-1}\theta(t-s)^{\alpha-1}
\xi_{\alpha}(\theta) Q((t-s)^{\alpha}\theta)L^{-1}f(s, x(s))d\theta ds\biggr\Vert_X\\
&\leq \left\Vert\int_{0}^{\delta}M^{-1}\xi_{\alpha}(\theta)Q(t^{\alpha}\theta)Mh(x)d\theta\right\Vert_{X} \\
&\quad +\alpha\left\Vert\int_0^t\int_0^{\delta}M^{-1}\theta(t-s)^{\alpha-1}\xi_{\alpha}
 (\theta) Q((t-s)^{\alpha}\theta)L^{-1}f(s, x(s))d\theta ds\right\Vert_X \\
&\quad +\alpha\left\Vert\int_{t-g}^t\int_{\delta}^{\infty}M^{-1}\theta(t-s)^{\alpha-1}
\xi_{\alpha}(\theta) Q((t-s)^{\alpha}\theta)L^{-1}f(s, x(s))d\theta ds\right\Vert_X\\
&\leq C_2M_{0}\int_{0}^{\delta}\xi_{\alpha}(\theta)d\theta
\Vert M\Vert\|h(x)\|_{L^{1/\beta}}\\
&\quad +C_1C_2M_0\alpha\left(\int_0^t(t-s)^{\frac{\alpha-1}{1-\beta}}ds\right)^{1-\beta}
\|f\|_{L^{1/\beta}} \int_0^{\delta}\theta\xi_{\alpha}(\theta)d\theta \\
&\quad +C_1C_2M_0\alpha\left(\int_{t-g}^t(t-s)^{\frac{\alpha-1}{1-\beta}}ds
\right)^{1-\beta}\|f\|_{L^{1/\beta}}
\int_{\delta}^{\infty}\theta\xi_{\alpha}(\theta)d\theta\\
&\leq C_2M_{0}\Vert M\Vert K_2\int_{0}^{\delta}\xi_{\alpha}(\theta)d\theta \\
&\quad +C_1C_2M_0K_{1}\alpha\left[\frac{1-\beta}{\alpha-\beta}\right]^{1-\beta}
\left(b^{\alpha-\beta}\int_0^{\delta}\theta\xi_{\alpha}(\theta)d\theta
 +\frac{1}{\Gamma(1+\alpha)}g^{\alpha-\beta}\right).
\end{align*}
From Definition~\ref{Definition 2.5} and Remark~\ref{Remark 2.2},
we deduce that the right-hand side of the last inequality tends to zero as
$g\to 0$ and $\delta\to 0$. Therefore, there are relatively
compact sets arbitrarily close to the set $\Pi(t)$, $t>0$. Hence, the set
$\Pi(t)$, $t>0$ is also relatively compact in $X$.

Since $T_{\varphi}$ is a convex compact metrizable subset of
$\omega$-$L^{1/\beta}(I, T)$, it suffices to prove the sequential continuity
of the map $S$. Let $\{u_{1,n}\}_{n\geq1},\ldots,\{u_{r,n}\}_{n\geq1}\subseteq T_{\varphi}$
be such that
\begin{equation}
\label{eq:3.5}
(u_{1,n},\ldots,u_{r,n})\to (u_1,\ldots,u_r)~
\text{in}~\omega-L^{1/\beta}(I, T),
~~ u_1,\ldots,u_r\in T_{\varphi}.
\end{equation}
Set $f_n:=f_n\left(\cdot,\cdot, B_1(\cdot)u_{1,n}(\cdot),\ldots,B_{r-1}(\cdot)u_{r-1,n}(\cdot)\right)$
and $h_n:=h_n\left(\cdot,B_{r}(\cdot)u_{r,n}(\cdot)\right)$. By the properties of the operator $H$
together with \eqref{eq:3.5}, we have
$H(f_n, h_n)\to H(f ,h)~\text{in}~ \omega-C(I, X)$,
where $(f_n, h_n)\to (f ,h)$ in $\omega-L^{1/\beta}(I\times X^{r}, Y)\times L^{1/\beta}(X^{2}, X)$, and
the limit functions are $f=f\left(\cdot,\cdot, B_1(\cdot)u_{1}(\cdot),\ldots,B_{r-1}(\cdot)u_{r-1}(\cdot)\right)$
and $h=h(\cdot,B_{r}(\cdot)u_{r}(\cdot))$.
Since $\{f_n\}_{n\geq1}$ and $\{h_n\}_{n\geq1}$ are bounded, there are
subsequences $\{f_{n_k}\}_{k\geq1}$ and $\{h_{n_k}\}_{k\geq1}$ of the sequences
$\{f_n\}_{n\geq1}$ and $\{f_n\}_{n\geq1}$, respectively, such that
$H(f_{n_k}, h_{n_k})\to z$ in $C(I,X)$ for some $z\in C(I,X)$. From
the fact that $H(f_n, h_n)\to H(f, h)~\text{in }\omega\text{-}C(I,X)$
and $H(f_{n_k}, h_{n_k})\to z~\text{in}~ C(I,X)$,
we obtain that $z=H(f, h)$ and $H(f_n, h_n)\to H(f, h)$ in $C(I,X)$.
Based on the definitions of operators $S$ and $H$, we know that
$S(u_1,\ldots,u_r)(t)=S_{\alpha}(t)Mx_0+H(f, h)(t)$.
According to the arguments above, we conclude that
$S\left(u_{1,n},\ldots,u_{r,n}\right)(t)\to S(u_1,\ldots,u_r)(t)$ in $C(I, X)$.
\hfill $\Box$
\end{proof}

Now, we consider the space $\overline{T}:=T\times \mathbb{R}$.
The elements of the space $\overline{T}$ will be denoted by
$\overline{u}_i:=(u_i, \tau_i)$, such that $u_i\in T$, $\tau_i\in \mathbb{R}$
and $i=1,\ldots,r$. The space $\overline{T}$ is endowed with the norm
$\Vert \overline{u}\Vert_{\overline{T}}=\max\lbrace\max(\Vert u_1\Vert_T, \vert \tau_1\vert),
\ldots,\max(\Vert u_r\Vert_T, \vert \tau_r\vert)\rbrace$. Then $\overline{T}$
is a separable reflexive Banach space. In view of \eqref{eq:2.2},
the norm on the space $L^{q}_{\omega}(I, {\overline{T}})$ becomes
\begin{multline*}
\Vert {\overline{u}}\Vert_\omega
=\sup\limits_{0\leq t_1\leq t_2\leq a}\Biggl\lbrace
\max\biggl\lbrace\max\left(\left\Vert\int_{t_1}^{t_2}u_1(s)ds\right\Vert_T,
\left\vert\int_{t_1}^{t_2}\tau_{1}(s)ds\right\vert\right),\\
\ldots, \max\left(\left\Vert\int_{t_1}^{t_2}u_r(s)ds\right\Vert_T,
\left\vert\int_{t_1}^{t_2}\tau_{r}(s)ds\right\vert\right)\biggr\rbrace\Biggr\rbrace.
\end{multline*}
Let the multivalued map $F: I\times X \rightrightarrows \overline{T}$ be defined by
\begin{equation}
\label{eq:F}
F(t, x)
:=\left\{(u_i, \tau_i)\in \overline{T}\Biggl\vert
\begin{array}{lll}
u_{1}\in U(t, x), \tau_1=g_{1}(t, x, u_{1})\\
\qquad \vdots\\
u_{r}\in U(t, x), \tau_r=g_{r}(t, x, u_{r})
\end{array} \right\}.
\end{equation}

\begin{lemma}
\label{Lemma 3.3}
The multivalued map $F$ given by \eqref{eq:F}
has bounded closed values and satisfies:
\begin{itemize}
\item[(1)] the map $t\to F(t, x)$ is measurable;
\item[(2)] $d_H(F(t, x), F(t, y))\leq l(t)\Vert x-y\Vert_X$ a.e., with $l\in L^{1}(I, \mathbb{R}^+)$;
\item[(3)] for any $\overline{u}_i=(u_i, \tau_i)\in F(t, x)$, we have
$\vert\tau_i\vert\leq a_4(t)+b_4(t)\Vert x\Vert_X+c_4(a_3(t)+c_3\Vert x\Vert_X)$
and $\Vert u_i\Vert_T\leq a_3(t)+c_3\Vert x\Vert_X$, $i=1,\ldots,r$.
\end{itemize}
\end{lemma}

\begin{proof}
From (H3.3) and (H4.3), we have the boundedness of $F(t, x)$. Moreover, item (3) also follows.
Since the graphs of functions $u_1\to g_1(t, x, u_1),\ldots,u_r\to g_r(t, x, u_r)$ are closed
on the set $U(t, x)$, we obtain the closedness of $F(t, x)$. The measurability of the
multivalued map is concluded and extended from \cite{AMA.20,AMA.40}. To prove item (2),
we consider $x, y\in X$, such that $x\neq y$, and any arbitrary $\epsilon_i >0$,
$i=1,\ldots,r$. Then, for each $u_i\in U(t, x)$, there exists $v_i\in U(t, y)$ satisfying
$\Vert u_i-v_i\Vert_T$ $\leq (k_3(t)+\epsilon_i)\Vert x-y\Vert_X$ and
$\vert g_i(t, x, u_i)-g_i(t, y, v_i)\vert\leq k_4(t)
\Vert x-y\Vert_X+k_{4}^{\prime}((k_3(t)+\epsilon_i)\Vert x-y\Vert_X)$.
From above, we get
$$
\max \Biggl\lbrace \sup d\biggl((u_1,g_1(t,x,u_1)), F(t,y)\biggr),
\ldots, \sup d\biggl((u_r,g_r(t,x,u_r)),
F(t,y)\biggr)\Biggr\rbrace\leq l(t)\Vert x-y\Vert_X,
$$
where $l(t):=\max \lbrace k_3(t), k_4(t)+k_{4}^{\prime}k_3\rbrace$.
Similarly, we can get
$$
\max \Biggl\lbrace \sup d\biggl((v_1,g_1(t,y,v_1)), F(t,x)\biggr),
\ldots, \sup d\biggl((v_r,g_r(t,y,v_r)),
F(t,x)\biggr)\Biggr\rbrace\leq l(t)\Vert x-y\Vert_X.
$$
We apply $\max$ between the two last $\max$-sets, to get our result.
\hfill $\Box$
\end{proof}

Let $\text{Eff} g^{**}(t, x)$ be the effective set, and $\text{Epi} g^{**}(t, x)$
the epigraph of functions $u_1\to g_{1}^{**}(t, x, u_{1})$,
$\ldots$, $u_r\to g_{r}^{**}(t, x, u_{r})$, that is,
\begin{itemize}
\item[(1)] $\text{Eff} g^{**}(t, x):=\lbrace u_i\in T: \max\lbrace
g_{1}^{**}(t, x, u_1),\ldots, g_{r}^{**}(t, x, u_r)\rbrace<+\infty\rbrace$,

\item[(2)] $\text{Epi} g^{**}(t, x):=\lbrace (u_i, \tau_i)\in \overline{T}: g_{1}^{**}(t, x, u_{1})
\leq \tau_1,\ldots,g_{r}^{**}(t, x, u_{r})\leq \tau_r\rbrace$.
\end{itemize}
Now, we present some properties of functions
$g_{i}^{**}(t, x, u_{i})$ via the following lemma.

\begin{lemma}
\label{Lemma 3.4}
For a.e. in $t\in I$, one has:
\begin{itemize}
\item[(1)] $\text{Eff} g^{**}(t, x)=\operatorname{cl} \operatorname{conv} U(t,x(t))$;

\item[(2)] $g_{1}^{**}(t, x, u_{1})+\cdots+g_{r}^{**}(t, x, u_{r})
=\min \lbrace \tau_{1}+\cdots+\tau_{r}\in\mathbb{R}:
(u_i,\tau_i)\in \operatorname{cl} \operatorname{conv} F(t, x)\rbrace$
for every $u_1,\ldots,u_r\in \text{Eff} g^{**}(t, x)$ and hence
$(u_i,g_{i}^{**}(t, x, u_{i}))\in \operatorname{cl} \operatorname{conv} F(t, x)$
when $u_i\in \operatorname{cl} \operatorname{conv} U(t, x)$ and $x\in X$;

\item[(3)] for any $\epsilon_i>0$, there exist closed sets $I_{\epsilon_{i}}\subseteq I,
\mu(I\setminus I_{\epsilon_{i}})\leq \epsilon_i$, such that
$(t, x, u_i)\to g_{i}^{**}(t, x, u_{i})$ are l.s.c. on $I_{\epsilon_{i}}\times X\times T$,
$i=1,\ldots,r$.
\end{itemize}
\end{lemma}

\begin{proof}
It is well known that the bipolar $g_{1}^{**}(t, x, u_{1}),\ldots,g_{r}^{**}(t, x, u_{r})$
are the $\Gamma$-regularization of $u_1\to g_{1,U}(t, x, u_1),\ldots,u_r\to g_{r,U}(t, x, u_r)$,
respectively. Let $x\in I$ a.e. be arbitrary. By (H4.3), each function of
$u_1\to g_{1,U}(t, x, u_1),\ldots,u_r\to g_{r,U}(t, x, u_r)$
has an affine continuous minorant. Then,
\begin{equation}
\label{eq:3.6}
\text{Epi} g^{**}(t, x)=\operatorname{cl} \operatorname{conv} \bigcup
\limits_{i=1}^{r}\text{Epi} g_{i,U}(t, x).
\end{equation}
Therefore, items (1) and (2) follow from \eqref{eq:3.6} and \eqref{eq:3.5}
(for more details see \cite{AMA.7}). Using Corollary~2.1 of \cite{AMA.40}
and items (1) and (2) of Lemma~\ref{Lemma 3.3},
for every $\epsilon_1,\ldots,\epsilon_r>0$, there are closed sets
$I_{\epsilon_{1}},\dots,I_{\epsilon_{r}}\subseteq I$ with
$\mu(I\setminus I_{\epsilon_{1}})\leq \epsilon_{1},\ldots,\mu(I\setminus I_{\epsilon_{r}})\leq \epsilon_r$
such that the map $\operatorname{cl} \operatorname{conv} F(t, x)$, restricted to $I_{\epsilon_{1}}\times X,\ldots,I_{\epsilon_{r}}\times X$,
has a closed graph in $I\times X\times \overline{T}$.
To show item (3), let us consider $(t_n, x_n, u_{i,n})_{n\geq1}\in I_{\epsilon_{i}}\times X\times T$,
such that $(t_n, x_n, u_{i,n})\to (t, x, u_{i})$.
If $\lim_{n\to \infty} g_{i}^{**}(t_n, x_n, u_{i,n})$ $=+\infty$,
then $g_{i}^{**}(t, x, u_{i})$ are l.s.c. at points $(t, x, u_i)$, $i=1,\ldots,r$.
If $\lim_{n\to \infty} g_{i}^{**}(t_n, x_n, u_{i,n})=\lambda_{i}$, with $\lambda_{i} \neq +\infty$,
then by using (H4.3), we get $\lambda_i\neq -\infty$. Hence, we can assume,
without loss of generality, that $g_{i}^{**}(t_n, x_n, u_{i,n})<+\infty$.
Then we have $(u_{i,n}, g_{i}^{**}(t_n, x_n, u_{i,n}))$
$\in \operatorname{cl} \operatorname{conv} F(t_n, x_n)$.
From the last formula, and based on the above, the map $\operatorname{cl} \operatorname{conv} F(t, x)$ restricted to
$I_{\epsilon_{1}}\times X,\ldots,I_{\epsilon_{r}}\times X$
has a closed graph in $I\times X\times \overline{T}$. We obtain that
$(u_i, \lambda_i)\in \operatorname{cl} \operatorname{conv} F(t, x)$. By the second item of this lemma,
we have $g_{i}^{**}(t, x, u_{i})\leq\lambda_i=\lim_{n\to \infty}g_{i}^{**}(t_n, x_n, u_{i,n}))$.
Consequently, the maps $(t, x, u_i)\to g_{i}^{**}(t, x, u_{i})$ are l.s.c.
on $I_{\epsilon_{i}}\times X\times T$.
\hfill $\Box$
\end{proof}


\section{Existence for Multiple Control Systems}
\label{sec:4}

In this section, we shall prove existence of solutions for the multiple control systems
\eqref{eq:1.1}--\eqref{eq:1.3} and \eqref{eq:1.1}--\eqref{eq:1.2},\eqref{eq:1.4}.
Let $\Lambda := S(T_\varphi)$. From Lemma~\ref{Lemma 3.2}, we
have that $\Lambda$ is a compact subset of $C(I,X)$.
It follows from \eqref{eq:3.2} and the definitions of $T_\varphi$ and $X_\varphi$
that $\mathcal{T}r_U\subseteq\mathcal{T}
r_{\operatorname{cl} \operatorname{conv} U}\subseteq\Lambda$.
Let the set-valued map $\overline{U}: C(I, X) \rightrightarrows 2^{L^{1/\beta}(I, T)}$ be defined by
\begin{equation*}
\overline{U}(x):=\left\{\theta_i: I\to T\text{ measurable}: \theta_i(t)\in U(t, x(t))
\text{ a.e.},~ i=1,\ldots,r\right\},
\quad x\in C(I, X).
\end{equation*}

\begin{theorem}
\label{Theorem 4.1}
The set $\mathcal{R}_U$ is nonempty and the set
$\mathcal{R}_{\operatorname{cl} \operatorname{conv} U}$ is a compact subset of the space
$C(J,X)\times\omega$-$L^{1/\beta}(I,T)$.
\end{theorem}

\begin{proof}
By hypotheses (H3.1) and (H3.2), we have that for any measurable
function $x: I\to X$, the map $t\to U(t, x(t))$ is measurable and has closed
values \cite[Proposition 2.7.9]{AMA.21}.
Therefore, it has measurable selectors \cite{AMA.20}.
So the operator $\overline{U}$ is well defined and its values are closed
decomposable subsets of $L^{1/\beta}(I, T)$. We claim that $x\to\overline{U}(x)$
is l.s.c. Let $x_* \in C(I, X)$, $\theta_{i,*}\in\overline{U}(x_*)$, $i=1,\ldots,r$,
and let $\{x_n\}_{n\geq1}\subseteq C(I, X)$ be a sequence converging to $x_*$.
It follows from \cite[Lemma 3.2]{AMA.48} that there are sequences
$\theta_{i,n}\in\overline{U}(x_n)$ such that
\begin{equation}
\label{eq:4.2}
\sum\limits_{i=1}^{r}\|\theta_{i,*}(t)-\theta_{i,n}(t)\|_T
\leq \sum\limits_{i=1}^{r}\biggl\lbrace d_T(\theta_{i,*}(t),
U(t,x_n(t)))+\frac{1}{in}\biggr\rbrace,
\quad \text{a.e. }t\in I.
\end{equation}
Since the map $y\to U(t,y)$ is $H$-continuous a.e. in $t\in I$ (by (H3.2)),
then a.e. in $t\in I$, the map $y\to U(t,y)$ is l.s.c.
\cite[Proposition 1.2.66]{AMA.21}. Hence,
each function $y\to d_T(\theta_{i,*}(t),U(t, y))$ is u.s.c. for a.e. $t\in I$.
It follows from \eqref{eq:4.2} that, a.e. in $t\in I$,
\begin{align*}
\lim_{n\to\infty}\sum\limits_{i=1}^{r}\|\theta_{i,*}(t)-\theta_{i,n}(t)\|_T
&\leq \lim_{n\to\infty}\sum\limits_{i=1}^{r}\sup d_T(\theta_{i,*}(t),U(t,x_n(t)))\\
&\leq  \sum\limits_{i=1}^{r}d_T(\theta_{i,*}(t),U(t,x_*(t)))=0.
\end{align*}
The last inequality together with \eqref{eq:3.2}
imply that $\theta_{i,n}\to \theta_{i,*}$ in
$L^{1/\beta}(I, T), ~i=1,\ldots,r$. Therefore,
the map $x\to\overline{U}(x)$ is l.s.c.
By \cite{AMA.42} (see also \cite[Theorem 2.8.7]{AMA.21}),
there exists a continuous function $m:\Lambda\to L^{1/\beta}(I, T)$ such that
$m(x)\in\overline{U}(x)$ for all $x\in \Lambda$.
Consider the map $\mathcal{P}:L^{1/\beta}(I, T)$ $\to L^{1/\beta}(I, T)$
defined by $\mathcal{P}(\theta_1,\ldots,\theta_r):=m(S(\theta_1,\ldots,\theta_r))$.
According to \eqref{eq:3.2} and the definition of $T_\varphi$, $\mathcal{P}(\theta_1,\ldots,\theta_r)\in T_\varphi$
for every $\theta_1,\ldots,\theta_r\in T_\varphi$. Due to Lemma~\ref{Lemma 3.2} and the continuity of $m$,
the map $\mathcal{P}: \omega$-$T_{\varphi}\to \omega$-$T_{\varphi}$ is continuous.
Since $\omega$-$T_{\varphi}$ is a convex metrizable compact set
in $\omega$-$L^{1/\beta}(I, T)$, by applying Schauder's fixed point theorem, we deduce
that this map has a fixed point $(\theta_{1,*},\ldots,\theta_{r,*})\in T^{r}_{\varphi}$,
that is, $(\theta_{1,*},\ldots,\theta_{r,*})
=\mathcal{P}(\theta_{1,*},\ldots,\theta_{r,*})=m(S(\theta_{1,*},\ldots,\theta_{r,*}))$.
Let $(u_{1,*},\ldots,u_{r,*}):=(\theta_{1,*},\ldots,\theta_{r,*})$ and
$x_* := S(\theta_{1,*},\ldots,\theta_{r,*})$. Then,
$(u_{1,*},\ldots,u_{r,*})=m(x_*)$ and $x_*=S(u_{1,*},\ldots,u_{r,*})$.
Thus, we have
\begin{align*}
x_*(t)&=S(u_{1,*},\ldots,u_{r,*})(t)=S_{\alpha}(t)M[x_{0}-h(x_*(t), B_{r}(t)u_{r,*}(t))]\\
&+\int_{0}^{t}(t-s)^{\alpha-1}T_{\alpha}(t-s)L^{-1}f(s, x_*(s),
B_{1}(s)u_{1,*}(s),\ldots, B_{r-1}(s)u_{r-1,*}(s))ds,
\end{align*}
$$
u_{1,*},\ldots,u_{r,*}\in U(t,x_*(t))\,\,\text{a.e.} \,\, t\in I,
$$
which imply that $(x_*(\cdot),u_{1,*}(\cdot),\ldots,u_{r,*}(\cdot))$
is a solution of the control system \eqref{eq:1.1}--\eqref{eq:1.3}.
Hence, $\mathcal{R}_U$ is nonempty. It is easy to see that
$\mathcal{R}_{\operatorname{cl} \operatorname{conv} U}\subseteq \Lambda\times T_{\varphi}$.
Since $\Lambda$ is compact in $C(I, X)$ and $T_{\varphi}$ is metrizable convex
compact in $\omega$-$L^{1/\beta}(I, T)$, we have that
$\mathcal{R}_{\operatorname{cl} \operatorname{conv} U}$ is relatively compact in
$C(I, X)\times\omega$-$L^{1/\beta}(I, T)$. Hence, to complete the proof of this
theorem, it is sufficient to prove that
$\mathcal{R}_{\operatorname{cl} \operatorname{conv} U}$
is sequentially closed in $C(I, X)\times\omega$-$L^{1/\beta}(I, T)$.
Let $\{(x_n(\cdot),u_{1,n}(\cdot),\ldots,u_{r,n}(\cdot)\}_{n\geq1}$
$\subseteq \mathcal{R}_{\operatorname{cl} \operatorname{conv} U}$
be a sequence converging to
$(x(\cdot),u_{1}(\cdot),\ldots,u_{r}(\cdot))$ in the space
$C(I, X)\times\omega$-$L^{1/\beta}(I, T)$. Then we have
$x_n=S(u_{1,n},\ldots,u_{r,n})$ and $(u_{1,n},\ldots,u_{r,n})$ $\to (u_{1},\ldots,u_{r})$
in $\omega$-$L^{1/\beta}(I, T)$. Denote $z := S(u_1,\ldots,u_r)$. From Lemma~\ref{Lemma 3.2},
we obtain that $z=x$, that is, $x$ is a solution of \eqref{eq:3.3}--\eqref{eq:3.4}
corresponding to $u_1,\ldots,u_r$. Hence, to prove that $(x(\cdot),u_{1}(\cdot),
\ldots,u_{r}(\cdot))\in \mathcal{R}_{\operatorname{cl} \operatorname{conv} U}$,
we only need to verify that $u_1,\ldots,u_r\in \operatorname{cl} \operatorname{conv} U(t, x(t))$ a.e. in $t\in I$.
Since $u_{1,n}\to u_1,\ldots,u_{r,n}\to u_r$ in $\omega$-$L^{1/\beta}(I, T)$,
by Mazur's theorem, we have
\begin{equation}
\label{eq:4.4}
u_{1}(t)\in\bigcap\limits_{n=1}^{\infty}
\operatorname{cl} \operatorname{conv} \left(\bigcup\limits_{k=n}^{\infty}
u_{1,k}(t)\right),\ldots,u_{r}(t)\in\bigcap\limits_{n=1}^{\infty}
\operatorname{cl} \operatorname{conv} \left(\bigcup\limits_{k=n}^{\infty}
u_{r,k}(t)\right)\quad \text{for a.e. }t\in I.
\end{equation}
From hypothesis (H3.2) and the fact that
$d_H(\operatorname{cl} \operatorname{conv} A,\operatorname{cl} \operatorname{conv} B)
\leq d_H(A,B)$ for sets $A,B$, the map $x \to \operatorname{cl} \operatorname{conv} U(t,x)$
is $H$-continuous. Then, from Proposition 1.2.86 in \cite{AMA.21}, we conclude that the map
$x\to\operatorname{cl} \operatorname{conv} U(t,x)$ has property $Q$. Therefore, we have
\begin{equation}
\label{eq:4.5}
\bigcap\limits_{n=1}^{\infty}\operatorname{cl} \operatorname{conv}
\left(\bigcup\limits_{k=n}^{\infty}\operatorname{cl} \operatorname{conv} U(t,x_k(t))\right)
\subseteq \operatorname{cl} \operatorname{conv} U(t,x(t))
\quad \text{for a.e. }t\in I.
\end{equation}
By \eqref{eq:4.4} and \eqref{eq:4.5}, we obtain
that $u_{1}(t),\ldots,u_{r}(t)\in \operatorname{cl} \operatorname{conv}
U(t,x(t))$ a.e. in $t\in I$. This means that
$\mathcal{R}_{\operatorname{cl} \operatorname{conv} U}$
is compact in $C(I, X)\times\omega$-$L^{1/\beta}(I, T)$.
\hfill $\Box$
\end{proof}


\section{Main Results}
\label{sec:5}

In order to state and prove our main results,
we firstly show the following helpful lemma.

\begin{lemma}
\label{Lemma 5.1}
For any function $x_*\in C(I, X)$ and any measurable selectors $u_{1,*},\ldots,u_{r,*}$
of the map $t\to \operatorname{cl} \operatorname{conv} U(t, x_*(t))$, there are sequences
$u_{1,n}(t),\ldots,u_{r,n}(t)$, $n\geq1$, of measurable
selectors of the map $t\to U(t, x_*(t))$, such that
\begin{equation}
\label{eq:5.1}
\sup_{0\leq t_1\leq t_2\leq a}\sum\limits_{i=1}^{r}\biggl\|\int_{t_1}^{t_2}(u_{i,*}(s)-u_{i,n}(s))ds
\biggr\|_T\leq\sum\limits_{i=1}^{r}\frac{1}{in},
\end{equation}
\begin{equation}
\label{eq:5.2}
\sup_{0\leq t_1\leq t_2\leq a}\sum\limits_{i=1}^{r}\biggl\vert
\int_{t_1}^{t_2}(g_{i}^{**}(s, x_*(s), u_{i,*}(s))
-g_{i}(s, x_*(s), u_{i,n}(s)))ds
\biggr\vert\leq\sum\limits_{i=1}^{r}\frac{1}{in}.
\end{equation}
The sequences $u_{i,n}$ converge to $u_{i,*}$ in $\omega$-$L^{1/\beta}(I, T)$.
\end{lemma}

\begin{proof}
Let $\overline{u}_{i,*}(t):=\left(u_{i,*}, g_{i}^{**}(t, x_*(t), u_{i,*}(t))\right)$, $i=1,\ldots,r$.
According to Lemma~\ref{Lemma 3.4}, we know that $\overline{u}_{i,*}(t)$
are measurable selectors of the map $t \to \operatorname{cl} \operatorname{conv} F(t, x_*(t))$.
From Lemma~\ref{Lemma 3.3}, the map $t\to F(t, x_*(t))$ is measurable and integrally bounded.
Hence, by using \cite[Theorem 2.2]{AMA.42}, we have that, for any $n\geq1$,
there exist measurable selections $\overline{u}_{1,n}(t),\ldots,\overline{u}_{r,n}(t)$
of the map $t\to F(t, x_*(t))$ such that
$$
\sup_{0\leq t_1\leq t_2\leq a}\biggl\|
\int_{t_1}^{t_2}(\overline{u}_{i,*}(s)-\overline{u}_{i,n}(s))ds
\biggr\|_{\overline{T}}\leq\frac{1}{in},
\quad i=1,\ldots,r.
$$
The definitions of $F$ and the weak norm on $L^{1/\beta}(I, \overline{T})$
give $\overline{u}_{i,n}(t)=(u_{i,n}, g_{i}(t, x_*(t), u_{i,n}(t)))$ and
$u_{i,n} \in U(t, x^*(t))$, $i=1,\ldots,r$, a.e. Then formulas \eqref{eq:5.1} and \eqref{eq:5.2} follow.
Hence, from Lemma~\ref{Lemma 2.1}, $u_{i,n}\to u_{i,*}$ in $\omega$-$L^{1/\beta}(I, T)$.
\hfill $\Box$
\end{proof}

\begin{theorem}
\label{Theorem 5.1}
Let any $(x_*(\cdot),u_{1,*}(\cdot),\ldots,u_{r,*}(\cdot))
\in \mathcal{R}_{\operatorname{cl} \operatorname{conv} U}$.
Then there exists a sequence
$$
(x_n(\cdot), u_{1,n}(\cdot),\ldots,u_{r,n}(\cdot))\in\mathcal{R}_{U},
\quad n\geq1,
$$
such that
\begin{equation}
\label{eq:5.3}
x_n\to x_* ~\text{in}~ C(I, X),
\end{equation}
\begin{equation}
\label{eq:5.4}
u_{i,n}\to u_{i,*}~ \text{in}~ L^{\frac{1}{\beta}}_{\omega}(I, T)~
\text{and}~ \omega\text{-}L^{\frac{1}{\beta}}(I, T),
\end{equation}
\begin{equation}
\label{eq:5.5}
\lim\limits_{n\to \infty}\sup_{0\leq t_1\leq t_2\leq a}
\sum\limits_{i=1}^{r}\biggl\vert\int_{t_1}^{t_2}(g_{i}^{**}(s, x_*(s), u_{i,*}(s))
-g_{i}(s, x_n(s), u_{i,n}(s)))ds
\biggr\vert=0.
\end{equation}
\end{theorem}

\begin{proof}
Let $(x_*(\cdot),u_{1,*}(\cdot),\ldots,u_{r,*}(\cdot))\in\mathcal{R}_{\operatorname{cl} \operatorname{conv} U}$.
From Lemma~\ref{Lemma 5.1}, for any $n\geq1$, there are measurable selectors $v_{1,n}(t),\ldots,v_{r,n}(t)$
of the multivalued map $t \rightrightarrows U(t,x_*(t))$ such that
\begin{equation*}
\sup_{0\leq t_1\leq t_2\leq a}\sum\limits_{i=1}^{r}\biggl\|
\int_{t_1}^{t_2}(u_{i,*}(s)-v_{i,n}(s))ds
\biggr\|_T\leq\sum\limits_{i=1}^{r}\frac{1}{in},
\end{equation*}
\begin{equation}
\label{eq:5.7}
\sup_{0\leq t_1\leq t_2\leq a}\sum\limits_{i=1}^{r}\biggl\vert
\int_{t_1}^{t_2}(g_{i}^{**}(s, x_*(s), u_{i,*}(s))-g_{i}(s, x_*(s), v_{i,n}(s)))ds
\biggr\vert\leq\sum\limits_{i=1}^{r}\frac{1}{in}.
\end{equation}
The sequences $v_{i,n}\to u_{i,*}$ in $\omega\text{-}L^{\frac{1}{\beta}}(I, T)$,
$i=1,\ldots,r$. For each fixed $n\geq1$, by (H3.2), we have that, for any $x\in X$ and a.e.
in $t\in I$, there exist $v_{i}\in U(t,x)$, $i=1,\ldots,r$, such that
\begin{equation}
\label{eq:5.8}
\|v_{i,n}(t)-v_i\|_T<k_i(t)\|x_*(t)-x\|_X+\frac{1}{in}.
\end{equation}
Let the map $\Upsilon_n: I\times X\to 2^T$ be defined by
\begin{equation}
\label{eq:5.9}
\Upsilon_n(t,x):=\left\{v_{i}\in T: v_{i},
i=1,\ldots,r, ~\text{satisfy inequality \eqref{eq:5.8}}\right\}.
\end{equation}
It follows from \eqref{eq:5.8} that $\Upsilon_n(t,x)$ is well defined a.e.
on $I$ and all $x\in X$, and its values are open sets. Using \cite[Corollary 2.1]{AMA.40}
(since we can assume, without loss of generality, that $U(t,x)$ is
$\Sigma\otimes\mathcal{B}_X$ measurable, see \cite[Proposition 2.7.9]{AMA.21}),
we obtain that, for any $\epsilon_1,\ldots,\epsilon_r>0$, there are compact sets
$I_{\epsilon_1},\ldots,I_{\epsilon_r}\subseteq I$ with
$\mu(I\backslash I_{\epsilon_{1}})\leq\epsilon_{1},\dots,\mu(I\backslash I_{\epsilon_{r}})\leq\epsilon_{r}$,
such that the restrictions of $U(t,x)$ to $I_{\epsilon_{1}}\times X,\ldots,I_{\epsilon_{r}}\times X$
are l.s.c and the restrictions of $v_{1,n}(t),\ldots, v_{r,n}(t)$ and $k_1(t),\ldots,k_r(t)$
to $I_{\epsilon_{1}},\ldots,I_{\epsilon_{r}}$, respectively, are continuous.
It means that \eqref{eq:5.8} and \eqref{eq:5.9} imply that the graphs of the
restrictions of $\Upsilon_n(t,x)$ to $I_{\epsilon_{i}}\times X$ are open sets in
$I_{\epsilon_{i}}\times X\times T$, $i=1,\ldots,r$, respectively.
Let the map $\Upsilon: I\times X\to 2^T$ be defined by
$\Upsilon(t,x):=\Upsilon_n(t,x)\cap U(t,x)$.
Clearly, a.e. in $t\in I$ and all $x\in X$, $\Upsilon(t,x)\neq\emptyset$.
Due to the arguments above and Proposition 1.2.47 in \cite{AMA.21}, we know that
the restrictions of $\Upsilon(t,x)$ to $J_{\epsilon_{i}}\times X$ are l.s.c. and
so does $\overline{\Upsilon}(t,x)=\overline{\Upsilon(t,x)}$.
Here the bar stands for the closure of a set in $T$.
Now consider the system \eqref{eq:1.1}--\eqref{eq:1.2}
with the constraint on the controls
\begin{equation}
\label{eq:5.11}
u_1(t),\ldots,u_r(t)\in\overline{\Upsilon}(t,x(t))
\quad \text{a.e. on } I.
\end{equation}
Since $\overline{\Upsilon}(t,x)\subseteq U(t,x)$, the estimate
of Lemma~\ref{Lemma 3.1} also holds in this matter.
Repeating the proof of Theorem~\ref{Theorem 4.1},
we obtain that there is a solution $(x_n(\cdot),u_{1,n}(\cdot),\ldots,u_{r,n}(\cdot))$
of the control system \eqref{eq:1.1}--\eqref{eq:1.2}, \eqref{eq:5.11}.
The definition of $\overline{\Upsilon}$ implies that
$(x_n(\cdot),u_{1,n}(\cdot),\ldots$, $u_{r,n}(\cdot))\in\mathcal{R}_U$ and
\begin{equation}
\label{eq:5.12}
\|v_{i,n}(t)-u_{i,n}(t)\|_T\leq k_i(t)\|x_*(t)-x_n(t)\|_X
+\frac{1}{in}, \quad i=1,\ldots,r.
\end{equation}
Since $(x_n(\cdot),u_{1,n}(\cdot),\ldots,u_{r,n}(\cdot))\in\mathcal{R}_U$, $n\geq1$, and
$(x_*(\cdot),u_{1,*}(\cdot),\ldots,u_{r,*}(\cdot))\in\mathcal{R}_{\operatorname{cl} \operatorname{conv} U}$,
we have
\begin{multline}
\label{eq:5.13}
x_*(t)=S_{\alpha}(t)M[x_{0}-h(x_*(t), B_{r}(t)u_{r,*}(t))]\\
+\int_{0}^{t}(t-s)^{\alpha-1}T_{\alpha}(t-s)L^{-1}f(s, x_*(s),
B_{1}(s)u_{1,*}(s),\ldots, B_{r-1}(s)u_{r-1,*}(s))ds
\end{multline}
and
\begin{multline}
\label{eq:5.14}
x_n(t)=S_{\alpha}(t)M[x_{0}-h(x_n(t), B_{r}(t)u_{r,n}(t))]\\
+\int_{0}^{t}(t-s)^{\alpha-1}T_{\alpha}(t-s)L^{-1}f(s, x_n(s),
B_{1}(s)u_{1,n}(s),\ldots, B_{r-1}(s)u_{r-1,n}(s))ds.
\end{multline}
Theorem~\ref{Theorem 4.1} and $\{(x_n(\cdot),u_{1,n}(\cdot)),
\ldots,u_{r,n}(\cdot))\}_{n\geq1}\subseteq\mathcal{R}_U\subseteq
\mathcal{R}_{\operatorname{cl} \operatorname{conv} U}$ imply that we can assume,
possibly up to a subsequence, that
$(x_n(\cdot),u_{1,n}(\cdot),\ldots,u_{r,n}(\cdot))\to(\overline{x}(\cdot),\overline{u}_{1}
(\cdot),\ldots$, $\overline{u}_{r}
(\cdot))\in\mathcal{R}_{\operatorname{cl} \operatorname{conv} U}$ in
$C(I,X)\times\omega$-$L^{1/\beta}(I, T)$. Subtracting \eqref{eq:5.14}
from \eqref{eq:5.13}, we obtain that
\begin{equation}
\begin{aligned}
\label{eq:5.15}
&\Vert x_*(t)-x_n(t)\Vert_X\\
&\leq \Vert S_{\alpha}(t)M[h(x_*(t), B_{r}(t)u_{r,*}(t))-h(x_*(t), B_{r}(t)v_{r,n}(t))]\Vert_X\\
&\quad +\Vert S_{\alpha}(t)M[h(x_*(t), B_{r}(t)v_{r,n}(t))-h(x_n(t), B_{r}(t)u_{r,n}(t))]\Vert_X\\
&\quad +\biggl\Vert\int_{0}^{t}(t-s)^{\alpha-1}
T_{\alpha}(t-s)L^{-1}[f(s, x_*(s), B_{1}(s)u_{1,*}(s),\ldots, B_{r-1}(s)u_{r-1,*}(s))\\
&\qquad -f(s, x_*(s), B_{1}(s)v_{1,n}(s),\ldots, B_{r-1}(s)v_{r-1,n}(s))]ds\biggr\Vert_X\\
&\quad +\biggl\Vert\int_{0}^{t}(t-s)^{\alpha-1}T_{\alpha}(t-s)L^{-1}[
f(s, x_*(s), B_{1}(s)v_{1,n}(s),\ldots, B_{r-1}(s)v_{r-1,n}(s))\\
&\qquad -f(s, x_n(s), B_{1}(s)u_{1,n}(s),\ldots, B_{r-1}(s)u_{r-1,n}(s))]ds\biggr\Vert_X.
\end{aligned}
\end{equation}
We use the previous estimations of our sufficient set of conditions,
together with the property of the operator $\Upsilon$ defined
in the proof of Lemma~\ref{Lemma 3.2}, and since $v_{i,n}\to u_{i,*}$,
$i=1,\ldots,r$, in $\omega\text{-}L^{\frac{1}{\beta}}(I, T)$
and $x_n\to \overline{x}$ in $C(I, X)$, then by letting
$n\to \infty$ in \eqref{eq:5.15} and realizing Lemma~\ref{Lemma 2.3},
we get $x_*=\overline{x}$, that is, $x_n\to x_*$ in $C(I, X)$. Hence,
from \eqref{eq:5.12}, we have $(v_{i,n}-u_{i,n})\to 0$ in $L^{\frac{1}{\beta}}(I, T)$.
Thus, $u_{i,n}=(u_{i,n}-v_{i,n})+v_{i,n}\to u_{i,*}$ in
$\omega\text{-}L^{\frac{1}{\beta}}(I, T)$ and in
$L_{\omega}^{\frac{1}{\beta}}(I, T)$. Hence, \eqref{eq:5.3} and \eqref{eq:5.4} hold.
Moreover, we have
\begin{equation}
\begin{aligned}
\label{eq:5.16}
\sup_{0\leq t_1\leq t_2\leq a}&\sum\limits_{i=1}^{r}\biggl\vert
\int_{t_1}^{t_2}(g_{i}^{**}(s, x_*(s), u_{i,*}(s))-g_{i}(s, x_n(s), u_{i,n}(s)))ds
\biggr\vert\\
&\leq\sup_{0\leq t_1\leq t_2\leq a}\sum\limits_{i=1}^{r}\biggl\vert
\int_{t_1}^{t_2}(g_{i}^{**}(s, x_*(s), u_{i,*}(s))-g_{i}(s, x_*(s), v_{i,n}(s)))ds
\biggr\vert\\
&\leq\sup_{0\leq t_1\leq t_2\leq a}\sum\limits_{i=1}^{r}\biggl\vert
\int_{t_1}^{t_2}(g_{i}(s, x_*(s), v_{i,n}(s))-g_{i}(s, x_n(s), u_{i,n}(s)))ds
\biggr\vert
\end{aligned}
\end{equation}
and assumption (H4.2) and \eqref{eq:5.12} give
$$
\vert g_{i}(t, x_*(t), v_{i,n}(t))-g_{i}(t, x_n(t), u_{i,n}(t))
\vert\leq (k^{\prime}_{4}(t)+k^{\prime\prime}_{4}k_i(t))\Vert x_*(t)
-x_n(t)\Vert_X+\frac{k^{\prime\prime}_{4}}{in}.
$$
Therefore, the last inequality together with \eqref{eq:5.7} and \eqref{eq:5.16}
imply that \eqref{eq:5.5} holds.
\hfill $\Box$
\end{proof}

\begin{theorem}
\label{Theorem 5.2}
Problem $(RP)$ has a solution and
\begin{equation}
\label{eq:5.17}
\min\limits_{(x, u_i)\in \mathcal{R}_{\operatorname{cl} \operatorname{conv} U}}J^{**}_{i}(x, u_i)
=\inf\limits_{(x, u_i)\in \mathcal{R}_{U}}J_{i}(x, u_i),
\quad i=1,\ldots,r.
\end{equation}
For any solution $(x_*, u_{1,*},\ldots,u_{r,*})$ of problem $(RP)$,
there exists a minimizing sequence
$$
(x_n, u_{1,n},\ldots,u_{r,n})\in \mathcal{R}_U, \quad n \geq 1,
$$
for problem $(P)$, which converges to $(x_*, u_{1,*},\ldots,u_{r,*})$ in the spaces
$C(I, X)\times \omega\text{-}L^{\frac{1}{\beta}}(I, T)$ and in
$C(I, X)\times L_{\omega}^{\frac{1}{\beta}}(I, T)$, and the following formula holds:
\begin{equation}
\label{eq:5.18}
\lim\limits_{n\to \infty}\sup_{0\leq t_1\leq t_2\leq a}\sum\limits_{i=1}^{r}\biggl\vert
\int_{t_1}^{t_2}(g_{i}^{**}(s, x_*(s), u_{i,*}(s))-g_{i}(s, x_n(s), u_{i,n}(s)))ds
\biggr\vert=0.
\end{equation}
Conversely, if $(x_n, u_{1,n},\ldots,u_{r,n})$, $n\geq1$, is a minimizing sequence for problem $(P)$,
then there is a subsequence $(x_{n_{k}}, u_{1,n_{k}},\ldots,u_{r,n_{k}})$, $k\geq1$, of the
sequence $(x_n, u_{1,n},\ldots,u_{r,n})$, $n\geq1$, and a solution $(x_*, u_{1,*},\ldots,u_{r,*})$
of problem $(RP)$ such that the subsequence $(x_{n_{k}}, u_{1,n_{k}},\ldots,u_{r,n_{k}}),$ $k\geq1$,
converges to $(x_*, u_{1,*},\ldots,u_{r,*})$ in $C(I, X)\times \omega\text{-}L^{\frac{1}{\beta}}(I, T)$
and relation \eqref{eq:5.18} holds for this subsequence $(x_{n_{k}}, u_{1,n_{k}},\ldots,u_{r,n_{k}})$, $k\geq1$.
\end{theorem}

\begin{proof}
By definition of functions $g_{i,U}(t, x, u_i)$, $i=1,\ldots,r$, (H3.3), (H4.3),
and the boundedness of the trajectories $\mathcal{T}r_{\operatorname{cl} \operatorname{conv}U}$
of the control system \eqref{eq:1.1}--\eqref{eq:1.2}, \eqref{eq:1.4} (Lemma~\ref{Lemma 3.1}),
we can get functions $m_i\in L^{1}(I, \mathbb{R}^{+})$ such that
\begin{equation}
\begin{aligned}
\label{eq:5.19}
&-m_i(t)=-[a_4(t)+b_4(t)L_0+c_4(a_3(t)+c_3L_0)]
\leq g_{i,U}(t, x, u_i), ~\text{a.e.}~t\in I,\\
&\text{with all}~x\in Q=\lbrace g\in X: \Vert g\Vert_X
\leq L_0\rbrace, \quad u_i\in U(t, x), \quad i=1,\ldots,r.
\end{aligned}
\end{equation}
Inequality \eqref{eq:5.19} and the properties of the bipolar (see \cite{AMA.15})
directly imply that
\begin{equation}
\label{eq:5.20}
-m_i(t)\leq g_{i}^{**}(t, x, u_i)\leq g_{i,U}(t, x, u_i),
\quad \text{a.e.}~t\in I, \quad x\in Q, \quad u_i\in T.
\end{equation}
Hence, from item (3) of Lemma~\ref{Lemma 3.4}, \eqref{eq:5.20},
and \cite[Theorem 2.1]{AMA.5}, the functional $J^{**}_{i}$, $i=1,\ldots,r$,
are lower semicontinuous on $\mathcal{R}_{\operatorname{cl} \operatorname{conv} U}\subseteq C(I, X)
\times \omega\text{-}L^{\frac{1}{\beta}}(I, T)$. Theorem~\ref{Theorem 4.1}
implies that $\mathcal{R}_{\operatorname{cl} \operatorname{conv} U}$ is compact in
$C(I, X)\times \omega\text{-}L^{\frac{1}{\beta}}(I, T)$. Therefore, problem $(RP)$
has a solution $(x_*, u_{1,*},\ldots,u_{r,*})$.
By item (1) of Lemma~\ref{Lemma 3.4}, we have
\begin{equation}
\label{eq:5.21}
J^{**}_{i}(x_*, u_{i,*})\leq\inf\limits_{(x, u_i)
\in \mathcal{R}_{U}}J_{i}(x, u_i),
\quad i=1,\ldots,r.
\end{equation}
Now, for every solution $(x_*, u_{1,*},\ldots,u_{r,*})$ of problem $(RP)$,
by using Theorem~\ref{Theorem 5.1}, we obtain that there exists a sequence
$(x_n, u_{1,n},\ldots,u_{r,n})\in \mathcal{R}_U$, $n\geq1$, such that
\eqref{eq:5.3}--\eqref{eq:5.5} hold. Since
\begin{equation}
\begin{aligned}
\label{eq:5.22}
\sum\limits_{i=1}^{r}&\biggl\vert\int_{I}(g_{i}^{**}(s,
x_*(s), u_{i,*}(s))-g_{i}(s, x_n(s), u_{i,n}(s)))ds
\biggr\vert\\
&\leq\sup_{0\leq t_1\leq t_2\leq a}\sum\limits_{i=1}^{r}\biggl\vert
\int_{t_1}^{t_2}(g_{i}^{**}(s, x_*(s), u_{i,*}(s))-g_{i}(s, x_n(s), u_{i,n}(s)))ds
\biggr\vert,
\end{aligned}
\end{equation}
by formulas \eqref{eq:5.5}, \eqref{eq:5.21} and \eqref{eq:5.22}, we get that
\eqref{eq:5.17}, \eqref{eq:5.18} hold and $(x_n(\cdot), u_{1,n}(\cdot),\ldots,u_{r,n}(\cdot))$
$\in \mathcal{R}_U$, $n\geq1$, is a minimizing sequence for problem $(P)$.
Let $(x_n(\cdot), u_{1,n}(\cdot),\ldots,u_{r,n}(\cdot))\in \mathcal{R}_U$, $n\geq1$,
be a minimizing sequence for problem $(P)$. According to Theorem~\ref{Theorem 4.1},
without loss of generality, we can assume that $(x_n, u_{1,n},\ldots,u_{r,n})$
$\to (x_*, u_{1,*},\ldots,u_{r,*})\in \mathcal{R}_{\operatorname{cl} \operatorname{conv} U}$
in $C(I, X)\times \omega\text{-}L^{\frac{1}{\beta}}(I, T)$ and
\begin{equation}
\label{eq:5.23}
\min (RP)=\lim_{n\to \infty}\sum\limits_{i=1}^{r}
\int_{I}g_{i}(s, x_n(s), u_{i,n}(s))ds.
\end{equation}
It follows from \eqref{eq:5.20} and the properties of function $g_{i}^{**}(t, x, u_i)$  that
\begin{equation}
\begin{aligned}
\label{eq:5.24}
\int_{I}g^{**}_{i}(s, x_*(s), u_{i,*}(s))ds
&\leq \lim_{n\to \infty}\inf\int_{I}g_{i}^{**}(s, x_n(s), u_{i,n}(s))ds\\
&\leq\lim_{n\to \infty}\int_{I}g_{i}(s, x_n(s), u_{i,n}(s))ds.
\end{aligned}
\end{equation}
From \eqref{eq:5.23} and \eqref{eq:5.24}, we obtain that
\begin{equation}
\label{eq:5.25}
\min (RP)=\sum\limits_{i=1}^{r}\int_{I}g_{i}^{**}(s, x_*(s), u_{i,*}(s))ds
=\lim_{n\to \infty}\sum\limits_{i=1}^{r}\int_{I}g_{i}(s, x_n(s), u_{i,n}(s))ds.
\end{equation}
Hence, $(x_*(\cdot), u_{1,*}(\cdot),\ldots,u_{r,*}(\cdot))
\in \mathcal{R}_{\operatorname{cl} \operatorname{conv} U}$
is a solution of problem $(RP)$. Hypotheses (H3.3) and (H4.3) and Lemma~\ref{Lemma 3.1},
imply that $\lbrace g_i(s, x_{n}(s), u_{i,n}(s))\rbrace_{n\geq1}$ is
uniformly integrable. Therefore, by the Dunford--Pettis theorem,
we have that there exists a subsequence $\lbrace g_i(s, x_{n_{k}}(s), u_{i,n_{k}}(s))\rbrace_{k\geq1}$
of the sequence $\lbrace g_i(s, x_{n}(s), u_{i,n}(s))\rbrace_{n\geq1}$ converging to certain functions
$\lambda_i(t)$ in the topology of the space $\omega\text{-}L^{1}(I, \mathbb{R})$.
Since $(u_{i,n_{k}}(s), g_i(s, x_{n_{k}}(s), u_{i,n_{k}}(s))\in F(s, x_{n_{k}}(s))$
a.e. in $s\in I$, Lemma~\ref{Lemma 3.3} implies that
$(u_{i,*}, \lambda_i(s))\in \operatorname{cl} \operatorname{conv} F(s, x_*(s))$, a.e. $s\in I$.
From Lemma~\ref{Lemma 3.4} we obtain that
\begin{equation}
\label{eq:5.26}
g_{i}^{**}(s, x_*(s), u_{i,*}(s))
\leq \lambda_i(s),~ \text{a.e.}~s\in I.
\end{equation}
Hence,
\begin{equation}
\label{eq:5.27}
\sum\limits_{i=1}^{r}\int_{0}^{t}g_{i}^{**}(s, x_*(s), u_{i,*}(s))ds
\leq\sum\limits_{i=1}^{r}\int_{0}^{t}\lambda_i(s)ds
=\lim_{k\to \infty}\sum\limits_{i=1}^{r}
\int_{0}^{t}g_{i}(s, x_{n_{k}}(s), u_{i,n_{k}}(s))ds
\end{equation}
for any $t\in I$. Now we can obtain from \eqref{eq:5.25}--\eqref{eq:5.27} that
$g_{i}^{**}(t, x_*(t), u_{i,*}(t))=\lambda_i(t)$, a.e. $t\in I$.
Hence the subsequence $g_i(s, x_{n_{k}}(s), u_{i,n_{k}}(s))\to g^{**}_{i}(s, x_{*}(s), u_{i,*}(s))$
as $k\to \infty$, in $\omega\text{-}L^{1}(I, \mathbb{R})$. This implies that
$$
\lim\limits_{k\to \infty}\sup_{0\leq t_1\leq t_2\leq a}\sum\limits_{i=1}^{r}\biggl\vert
\int_{t_1}^{t_2}(g_{i}^{**}(s, x_*(s), u_{i,*}(s))-g_{i}(s, x_{n_{k}}(s), u_{i,n_{k}}(s)))ds
\biggr\vert=0.
$$
Hence, we proved that \eqref{eq:5.18} holds for the subsequence $(x_{n_{k}}, u_{i,n_{k}})$, $k\geq1$.
\hfill $\Box$
\end{proof}


\section{Conclusions}
\label{sec:conc}

We studied optimality and relaxation of multiple control problems,
described by Sobolev type nonlinear fractional differential equations
with nonlocal control conditions in Banach spaces.
The optimization problems were defined by multi-integral functionals
with integrands that are not convex in the controls, subject
to control systems with mixed nonconvex constraints on the controls.
We proved appropriate sufficient conditions assuring existence
of optimal solutions for the relaxed problems. Moreover, we have shown,
in suitable topologies, that the optimal solutions are limits of minimizing
sequences of systems with respect to the trajectory, multi-controls,
and the functional.


\begin{acknowledgements}
This research was initiated during a visit of Debbouche and Torres
to the Department of Mathematical Analysis of the University of Santiago of Compostela,
Spain, followed by a visit of Debbouche to the Department of Mathematics of
University of Aveiro, Portugal. The hospitality and the financial support
provided by the host institutions in Spain and Portugal are here gratefully acknowledged.
Nieto has been partially supported by the Ministerio de Econom\'{\i}a y Competitividad
of Spain under grants MTM2010--15314 and MTM2013--43014--P, Xunta de Galicia
under grant R2014/002, and co-financed by the European Community fund FEDER.
Torres was supported by funds through The Portuguese Foundation for Science and Technology (FCT),
within CIDMA project UID/MAT/04106/2013 and OCHERA project PTDC/EEI-AUT/1450/2012,
co-financed by FEDER under POFC-QREN with COMPETE reference FCOMP-01-0124-FEDER-028894.
\end{acknowledgements}



\end{document}